\documentclass{amsart}

\newcommand{\hl}[1]{#1}

\usepackage{tikz}
\usetikzlibrary{intersections,patterns}

\newtheorem{theorem}{Theorem}[section]
\newtheorem{lemma}[theorem]{Lemma}

\theoremstyle{definition}
\newtheorem{definition}[theorem]{Definition}

\theoremstyle{remark}

\numberwithin{equation}{section}

\usepackage{amsaddr}
\usepackage{subcaption}
\usepackage{amsmath}
\usepackage{amsfonts}
\usepackage{amssymb}
\usepackage{dsfont}
\usepackage{epstopdf}
\usepackage{graphicx}
\usepackage{enumerate}
\usepackage{xfrac}
\usepackage{amstext}    % defines the \text command, needed here
\usepackage{array} 
\usepackage{todonotes}

\numberwithin{equation}{section}
\numberwithin{figure}{section}

\usetikzlibrary{patterns}

\newcommand{\ab}{\mathbf{a}}
\newcommand{\pp}{\mathbf{p}}
\newcommand{\kk}{\mathbf{k}}

\newcommand{\rem}[1]{}

\begin{document}

\title{Instability of Equilibria for the 2D Euler equations on the torus}


%\author{nobody}

% \author{Holger R.~Dullin\thanks{School of Mathematics and Statistics, Carslaw Building (F07), The University of Sydney, NSW 2006. \mbox{Holger.Dullin@sydney.edu.au}}  \and Robert Marangell\thanks{School of Mathematics and Statistics, Carslaw Building (F07), The University of Sydney, NSW 2006. \mbox{Robert.Marangell@sydney.edu.au}} \and Joachim~Worthington\thanks{Corresponding Author. School of Mathematics and Statistics, Carslaw Building (F07), The University of Sydney, NSW 2006. \mbox{Joachim.Worthington@sydney.edu.au}} 
% }
% 
% \authorrunning{H. Dullin, R. Marangell, J. Worthington} % if too long for running head

% \affiliation{% J. Worthington (Corresponding Author) % \at
%              School of Mathematics and Statistics \\
% 					Carslaw Building (F07) \\
% 					The University of Sydney, NSW 2006 \\
%              \email{Joachim.Worthington@sydney.edu.au}   (Corresponding Author)          \\
%           H. R. Dullin  % \at
%              School of Mathematics and Statistics \\
% 					Carslaw Building (F07) \\
% 			The University of Sydney, NSW 2006 \\
%              {Holger.Dullin@sydney.edu.au}     \\
% 	\and
%          R. Marangell % \at
%              {Robert.Marangell@sydney.edu.au}     \\
%              School of Mathematics and Statistics \\
% 					Carslaw Building (F07) \\
% 					The University of Sydney, NSW 2006 

\author{Holger Dullin, Robert Marangell and Joachim Worthington}
\address{School of Mathematics and Statistics,
					Carslaw Building (F07),\\
			The University of Sydney, NSW 2006\\
	Email: Joachim.Worthington@sydney.edu.au (corresponding author),\\Holger.Dullin@sydney.edu.au, Robert.Marangell@sydney.edu.au.  }

\maketitle

\begin{abstract}

We consider the hydrodynamics of an incompressible fluid on a 2D periodic domain. There exists a family of stationary solutions with vorticity given by $\Omega^*=\alpha\cos (\mathbf{p} \cdot  \mathbf{x} )+\beta \sin (\mathbf{p} \cdot  \mathbf{x} )$. 
This situation can be approximated as a structure preserving finite dimensional Hamiltonian system by a truncation introduced by \cite{Zeitlin90,Zeitlin05} or by the more standard Galerkin style finite element method. %want a citation??
We use these two truncations to analyse the linear stability of these solutions and analytical and numerical results are compared. Following the methods used by \cite{Li00} the problem is divided into subsystems and we prove that most subsystems are linearly stable. We derive a sufficient condition 
for a subsystem to be linearly unstable and derive an explicit lower bound for the associated real eigenvalues independent of the truncation size $N$. 
Then we show that the corresponding eigenvectors are in $\ell^2$.
This together with known stability results for the 2D periodic Euler equations allows us to conclude that most of these stationary solutions are nonlinearly unstable. We confirm our results with a numerical computation of the spectrum for a large, finite truncation. Finally we discuss the essential spectrum of the full problem as the limit of the truncated problem.

%\keywords{Euler Equation \and Stability Theory \and Hydrodynamics}
\end{abstract}

\section{Introduction}
\label{intro}

In terms of the vorticity $\Omega(\mathbf{x},t):(\mathds{T}^2 \times \mathds{R}^{+}) \to \mathds{R}$, the 2D incompressible Euler equations are  (see \cite{Arnold78} Appendix 2 for an overview)

\begin{equation}
\label{eq:originalsystem}
\frac{\partial \Omega}{\partial t} +u_1 \frac{\partial \Omega}{\partial x_1}+u_2\frac{\partial \Omega}{\partial x_2}=0,\;\;\;\;\; \frac{\partial u_1}{\partial x_1}+\frac{\partial u_2}{\partial x_2}=0.
\end{equation}

Here $\mathbf{x}=(x_1,x_2)^T$ and $u_1$, $u_2$ are the velocity components in the $x_1$ and $x_2$ directions respectively. We impose periodic boundary conditions $\Omega(\pi,x_2,t)=\Omega(-\pi,x_2,t)$ and $\Omega(x_1,\pi,t)=\Omega(x_1,-\pi,t)$.  There is a family of stationary solutions given by 
\begin{equation}
\label{eq:mainequilibria}
	\Omega^*=\alpha \cos (\mathbf{p} \cdot  \mathbf{x} )+\beta \sin (\mathbf{p} \cdot  \mathbf{x} )
\end{equation}
for $\alpha,\beta\in\mathds{R}$ and $\mathbf{p}\in\mathds{Z}^2$. 

The study of stability of certain solutions of the planar Euler equations was initiated 
by the seminal work by \cite{Arnold66} on the Lie-Poisson structure of the Euler equations where he
invented the Energy-Casimir method to prove stability. This is revisited in \cite{Arnold98}, 
in particular Section II.4. Arnol'd discusses a slightly more general problem where the torus has 
dimensions $X \times 2 \pi$ and $\mathbf{p} = (0,1)^T$, and shows that this solution is non-linearly 
stable when $X \le 2\pi$. 
In \cite{Meshalkin61} it was shown that any equilibrium with $\mathbf{p} = (0,1)^T$ and $X > 2\pi$ is linearly unstable for the viscous problem using linear stability analysis and infinite continued fractions. That paper also shows linear instability for $\mathbf{p} = (0, m)^T$, $m>1$ and any $X$. 
The linear instability of the inviscid problem for $\mathbf{p} = (0,m)^T$, $m>1$ was proved in \cite{Friedlander99} 
and discussed in \cite{Butta10}. Under the condition $m^2\neq m_1^2+m_2^2$ for any positive integers $m_1,m_2$, it was shown in \cite{Friedlander97} that the steady state for $\mathbf{p} = (0,m)^T$ is nonlinearly unstable.

In \cite{Li00} it is shown how to block-diagonalise the linearisation about the equilibrium with general $\mathbf{p}$ into so-called `classes',
and using this approach he again showed that $p=(0,1)^T$ is Lyapunov stable. 
This is used in \cite{Li04} where the essential and discrete spectrum of the linearisation of \eqref{eq:originalsystem} are studied at the steady state \eqref{eq:mainequilibria}. They studied the full infinite system, approaching the problem from a functional analytic perspective. They found an upper bound on the number of non-imaginary isolated eigenvalues, and described the essential spectrum. 
Furthermore, they showed that the spectrum of the linearised operator is the union of the spectrum coming from each of the classes from \cite{Li00} 
and that the spectral mapping theorem holds for the  Euler operator linearised about $\Omega^*$ (Theorem 2 in \cite{Li04}).

In the viscous problem solutions $e^{-\nu m^2 t}\cos (mx)$ are called \emph{bar states} in \cite{Beck12}.
%  and are a particular example of a \emph{Taylor-Green vortex} (see \cite{Madja} Section 5.1.1). 
They show that for non-zero viscosity $\nu$, and $m=1$ these bar states are `quasi-stationary', in that they decay on a slow timescale depending on the viscosity. 

We combine the block-diagonalisation used by \cite{Li00} with the  structure preserving finite-dimensional sine truncation \cite{Zeitlin90} and the Galerkin (see \cite{marchuk75}) finite element truncation to prove that for a large class of $\mathbf{p}$ the stationary solutions \eqref{eq:mainequilibria} are nonlinearly
 unstable.
Zeitlin's sine truncation leads to a finite dimensional Poisson structure and 
the Hamiltonian structure of the original PDE and its Casimirs are preserved in this finite-dimensional truncation. The Galerkin truncation does not preserve these Casimirs.
See \cite{Arnold98} and \cite{Kolev07} for a discussion of the use of Poisson brackets in hydrodynamics. The theoretical background of Zeitlin's sine truncation (and a related truncation for a spherical domain) is discussed in \cite{hoppe89}, and the concept of a ``limit'' of this algebra is discussed in depth in \cite{hoppe91}.

In Section \ref{sec:evolution}, the problem and the associated notation are introduced. The system is first decomposed into Fourier modes which are described by a non-canonical infinite dimensional Hamiltonian system. Then truncation is taken to reduce to a finite-mode approximation. We linearise around the steady state, which decouples the problem into subsystems.

In Section \ref{sec:Instability}, we reproduce the ``stable disc theorem'' from \cite{Li00} in the truncated setting. 
This theorem states that for classes whose mode numbers $\mathbf{a}$ satisfy $|\mathbf{a}|>|\mathbf{p}|$ the spectrum is stable. Thus most class subsystems do not contribute unstable modes to the spectrum of the full operator.
 Then we prove our ``unstable disc theorem"~\ref{thm:realevs}, which states that if exactly one mode number
 of a given class is inside the unstable disc then for sufficiently large $N$ there is a positive real eigenvalue.
 %We then consider the classes in a general way and find a lower bound for positive real solutions of the eigenvalue problem, under some conditions. 
In our fundamental Theorem \ref{thm:main}  we show that with certain additional assumptions this real eigenvalue is bounded away from zero when $N \to \infty$.
Furthermore in Lemma~\ref{lem:eigenvector} we show that the corresponding eigenvector of the infinite dimensional system is in $\ell^2$.

Section \ref{sec:Classes} provides the main Theorem~\ref{thm:fin} demonstrating non-linear instability of the stationary solution \eqref{eq:mainequilibria} for all choices of $\mathbf{p}$ but a few exceptions. 
In Lemma~\ref{lem:RealitySatisfied} we establish when the conditions for the lower bound needed in Section \ref{sec:Instability} are met. 
Zeitlin's truncation requires some care when proving results for both stable and unstable classes. Specifically it is not clear that the intersections between our classes and the disc $|\mathbf{a}|<|\mathbf{p}|$ behave in the way we expect.  In Lemma~\ref{lem:NSequence} we show that for most choices of $\mathbf{p}$, there is an appropriate truncation size $N$ to control the Zeitlin truncation so that Theorem \ref{thm:main} can be applied. The other cases of $\mathbf{p}$ can be treated using the Galerkin trunction.

Section \ref{sec:numerical} provides some numerical results. The numerical efficiency and accuracy of Zeitlin's truncation is compared favourably to the Galerkin truncation. A connection is made between the nature of the subsystems and the number and type of non-imaginary eigenvalues. A discussion of the number of non-imaginary eigenvalues is included, and the accuracy of our calculated lower bound is assessed. A brief section on the pure imaginary spectrum of our finite mode systems is included, replicating the results in \cite{Li04} via a very different method. % By taking some approximations  
We can show that the pure imaginary spectrum of our finite dimensional approximation approaches the essential spectrum of the full system. 
As a result we can naturally define a density of eigenvalues in the essential spectrum.

%NEED TO RETURN TO THIS IN CONTEXT OF ANY NEW STUFF.

\section{Vorticity Evolution in Fourier Space, Truncation, Linearisation}
\label{sec:evolution}
\subsection{Hamiltonian Formulation}
\label{sec:setup}

The stream function $\Psi$ is defined through its relation to the fluid velocities by
\begin{equation}
	u_1=+\frac{\partial \Psi}{\partial x_2},\;\;\;\;u_2=-\frac{\partial \Psi}{\partial x_1}.
\end{equation}
The relationship between the stream function and the vorticity is
\begin{equation}
	\Omega=-\nabla^2 \Psi 
	\label{eq:voriticystream}
\end{equation}
and hence the PDE can be written as
\begin{equation} 
	\frac{\partial \Omega}{\partial t}=\frac{\partial \Omega}{\partial x_1}\frac{\partial \Psi}{\partial x_2}-\frac{\partial \Omega}{\partial x_2}\frac{\partial \Psi}{\partial x_1}.
	\label{eq:omegade} 
\end{equation}

For a fixed $\mathbf{p} \in \mathds{Z}^2$ and $\Gamma \in \mathds{R}$ we wish to analyse the steady state $\Omega^* = \alpha\cos (\mathbf{p} \cdot  \mathbf{x} )+\beta\sin (\mathbf{p} \cdot  \mathbf{x} )$. 

Note that we can write $\Omega^*=2\Gamma\cos(\mathbf{p}\cdot\mathbf{x}+\theta)$, where $\theta=\pm\tan^{-1}\left ( \frac{-\beta}{\alpha}\right )$ and $\Gamma=\pm\frac{\sqrt{\alpha^2+\beta^2}}{2}$. The signs of $\theta$ and $\Gamma$ will depend on the signs of $\alpha$ and $\beta$. If $\alpha=0$, then take $\theta=\frac{\pi}{2}$. Define $\mathbf{c}=\frac{\theta}{|\mathbf{p}|^2}\mathbf{p}$, so $\Omega^*=2\Gamma\cos(\mathbf{p}\cdot(\mathbf{x}+\mathbf{c}))$. Thus by taking the translation $\tilde{\mathbf{x}}=\mathbf{x}+\mathbf{c}$ we can instead just consider the steady state
$\Omega^*=2{\Gamma}\cos(\mathbf{p}\cdot\tilde{\mathbf{x}})$ by a change of origin. Therefore for the remainder of this paper we drop the tilde and simply consider the steady state
$$\Omega^*=2{\Gamma}\cos(\mathbf{p}\cdot{\mathbf{x}}).$$

Expand $\Omega$ into a Fourier series with coefficients $\omega_{\mathbf{k}}(t)$ as 
$$\Omega(\mathbf{x},t)=\sum_{\mathbf{k}\in\mathds{Z}^2} \omega_{\mathbf{k}}(t)e^{i\mathbf{k}\cdot{\mathbf{x}}}$$ and combine \eqref{eq:voriticystream} and \eqref{eq:omegade}. Then the Fourier coefficients are governed by the ODEs
\begin{equation}
	\dot{\omega}_{\mathbf{k}}(t)=\sum_{\mathbf{l}\in\mathds{Z}^2 \backslash \{\mathbf{0}\}}\frac{\mathbf{k} 				\times \mathbf{l}}{\;|\mathbf{l}|^2}\omega_{-\mathbf{l}}\omega_{\mathbf{k}+\mathbf{l}}
\label{eq:fourierodes}
\end{equation}
(where $\mathbf{x}	\times \mathbf{y}=x_1y_2-x_2y_1$ for $\mathbf{x}, \mathbf{y} \in\mathds{R}^2$, and $\dot{\omega}_{\mathbf{k}}:=\frac{\mathrm{d}}{\mathrm{d}t}({\omega}_{\mathbf{k}})$).
The condition $\omega_{\mathbf{k}}=\bar{\omega_{-\mathbf{k}}}$ is necessary for $\Omega$ to be real.

Define the `ideal fluid' Poisson Bracket in Fourier Space as
\begin{equation}
	\{f,g\}=\sum_{\mathbf{k},\mathbf{l}} \frac{\partial f}{\partial \omega_{\mathbf{k}}}\frac{\partial g}{\partial \omega_{\mathbf{l}}}(\mathbf{k}	\times \mathbf{l})\omega_{\mathbf{k}+\mathbf{l}}.
	\label{eq:Bracket}
\end{equation}
The corresponding infinite dimensional Poisson structure matrix  is 
 $J_{\mathbf{k},\mathbf{l}}=(\mathbf{k}	\times\mathbf{l})\omega_{\mathbf{k}+\mathbf{l}}$. Then \eqref{eq:fourierodes}  is a non-canonical Hamiltonian system with corresponding Hamiltonian
\begin{equation}
\mathcal{H}=\sum_{\mathbf{k}\in \mathds{Z}^2\backslash \{\mathbf{0}\}} \frac{\omega_{+\mathbf{k}}\omega_{-\mathbf{k}}}{|\mathbf{k}|^2}
	=\frac{1}{2}\sum_{\mathbf{k}\in \mathds{Z}^2\backslash \{\mathbf{0}\}} \frac{{\omega_{+\mathbf{k}}}^2}{|\mathbf{k}|^2}.
	\label{eq:untruncatedhamiltonian}
\end{equation}
The Hamiltonian is obtained from the Kinetic energy 
$$\mathcal{H}=\frac{1}{2}\int ||\mathbf{u}||^2 \mathrm{d}\mathbf{x}=-\frac{1}{2}\int \Omega \Psi \mathrm{d}\mathbf{x}\, .$$ 
% see e.g. \cite{Newton01}.
%{HRD: do we need to assume that $\omega_0 = 0$ for $\Omega$ to come from $\Psi$?}

\subsection{Galerkin Truncation}
\label{sec:galerkintruncation}

We now truncate to a finite mode approximation and study the spectrum of the equilibrium
corresponding to $\Omega^*$. We will present two approaches to this: a Galerkin-style finite element truncation, and a more sophisticated Poisson structure truncation by Zeitlin.

First consider the Galerkin-style truncation (see \cite{Ray11}). Define the domain for our truncated Fourier modes 
\begin{equation}\mathcal{D}=[-N,N]^2\cap \mathds{Z}^2.
	\label{eq:trundomain}
\end{equation}
 
Now set $\omega_\mathbf{k}=0, \;\dot{\omega}_\mathbf{k}=0$ for all $\mathbf{k}\not{\in}\mathcal{D}$. Then the differential equations \eqref{eq:fourierodes} define a finite set of ODEs, but not a Poisson system.

\subsection{Zeitlin's Truncation}
\label{sec:zeitlintruncation}

An alternative truncation is that described by \cite{Zeitlin90} (see also \cite{Pope90}, \cite{hoppe89}).
Restrict to the set of Fourier modes to
\begin{equation}
	\omega_\mathbf{k},\;\;\mathbf{k}\in\mathcal{D},
\end{equation}
and whenever a mode is referenced that is outside the domain $\mathcal{D}$ it is mapped back into $\mathcal{D}$.
For this we have the notation $\hat{\mathbf{k}}$, which for any $\mathbf{k}$ denotes a mode $\hat{\mathbf{k}} \in \mathcal{D}$
for which the difference $\mathbf{k} - \hat{\mathbf{k}}=(2N+1)(a,b)^t$ for some integers $a,b$.

Zeitlin gave the following Poisson bracket on the domain $\mathcal{D}$:
\begin{align}
\label{eq:zeitlinstructure}
\{f,g\}&=\sum_{\mathbf{k},\mathbf{l}\in\mathcal{D}} \frac{\sin (\varepsilon \mathbf{k} 	\times \mathbf{l} )}{\varepsilon}\frac{\partial f}{\partial \omega_{\mathbf{k}}}\frac{\partial g}{\partial \omega_{\mathbf{l}}} \omega_{\widehat{\mathbf{k}+\mathbf{l}}}, \\
J_{\mathbf{k},\mathbf{l}}&=\frac{1}{\varepsilon}\sin (\varepsilon \mathbf{k} 	\times \mathbf{l} ) \omega_{\widehat{\mathbf{k}+\mathbf{l}}}
\end{align} 
where $\mathbf{k},\mathbf{l}\in\mathcal{D}$, and $\varepsilon=\frac{2\pi}{2N+1}$. The corresponding truncation of \eqref{eq:untruncatedhamiltonian} is the Hamiltonian 
\begin{equation}
\mathcal{H}=\frac{1}{2}\sum_{\mathbf{k}\in \mathcal{D}\setminus \{\mathbf{0}\}} \frac{\omega_{+\mathbf{k}}\omega_{-\mathbf{k}}}{|\mathbf{k}|^2}\,,
\end{equation}
where only the domain of summation has changed. 

The vector field under the Zeitlin truncation is thus given by
\begin{align}
\dot \omega_{\mathbf{k}} = 
	\left(J \nabla \mathcal{H}\right)_{\mathbf{k}}&=\sum_{\mathbf{l}\in D} J_{\mathbf{k},\mathbf{l}}\nabla H_\mathbf{l} \\
		&=\frac{1}{\varepsilon}\sum_{\mathbf{l}\in D}\sin (\varepsilon\mathbf{k}\times \mathbf{l})
		\omega_{\widehat{\mathbf{k}+\mathbf{l}}}\frac{\omega_{-\mathbf{l}}}{|\mathbf{l}|^2}.
		\label{eq:zeitlinode}
\end{align}

%The Lie-Poisson bracket \eqref{eq:zeitlinstructure} corresponds to the Lie Algebra of functions on the torus
%\begin{equation}
%\label{eq:liealgebra}
%[T_{\mathbf{n}},T_{\mathbf{m}}]=\frac{1}{\varepsilon}\sin \left ( \varepsilon\mathbf{n}\times \mathbf{m} \right ) T_{\mathbf{n}+\mathbf{m}\;\text{mod}\;\;2N+1}
%\end{equation}
%with the basis $T_{\mathbf{n}}=e^{i\mathbf{n}\cdot \mathbf{x}}$. 
%%\footnote{HRD: what does this mean? As a differential operator? Then multiply by $n \partial$. Even then it would only give the infinite dimensional algebra, not the finite dimensional one. It is not easy to write down $L_n$, see appendix of Zeilin91. BTW, in the paper Pope90 they point out that when $N$ is even on can do almost the same, except with an additional minus sign...}

The primary theoretical advantage of the Zeitlin truncation over the Galerkin truncation is that $2N+1$ of the Casimirs present in the full system are preserved in the Zeitlin truncated system. The disadvantage of the Zeitlin truncation is that we must take some care when making arguments based on it (see Section \ref{sec:Classes}). The two truncations will be compared both numerically and analytically in later sections. For details of the construction and a description of the Casimirs see \cite{Zeitlin90,Zeitlin05}.

%The $N$ Casimirs as described in \cite{Zeitlin05} are given by 
%\begin{equation}
%\label{eq:casimirs}
%C_M=\sum_{\mathcal{I}^M}\omega_{\mathbf{i}_1}...\omega_{\mathbf{i}_M}\cos \left ( \frac{4\pi}{N}A(\mathbf{i}_1,...,\mathbf{i}_M) \right )
%\end{equation}
%where $M\leq 2N$ and
%\begin{equation}
%	A(\mathbf{i}_1,...,\mathbf{i}_M)=\frac{1}{2}\sum_{1\leq j<k\leq M} \mathbf{i}_k	\times \mathbf{i}_j
%\end{equation}
% is the area spanned by the index vectors. The summation is taken over index sets $%\mathcal{I}^M\subset \mathds{D}$ with $|\mathcal{I}^M|=M$. As $N\to\infty$, these converge to the known Casimirs for the full system.

\subsection{The Linearised system}

For a fixed $\mathbf{p}\in \mathcal{D}\setminus \{\mathbf{0}\}$ the equilibrium point of the PDE $\Omega^* = 2 \Gamma \cos (\mathbf{p} \cdot  \mathbf{x} )$ is an equilibrium 
point of the truncated ODE given by
\begin{equation}
\omega_{\mathbf{l}}^*=
\left\{
	\begin{array}{ll}
		\Gamma  & \mbox{if } \mathbf{l}=\pm\mathbf{p} \\
		0 & \mbox{otherwise.}
	\end{array}
\right .
\label{eq:fourierequil}
\end{equation}
%As mentioned in Section \ref{sec:setup}, only consider $\Gamma\in \mathds{R}$ ($\Gamma\in\mathds{C}$ is treated in Li \cite{Li00} but this is equivalent to $\Gamma\in\mathds{R}$). 

As the original problem has symmetries $x \leftrightarrow -x$, $y \leftrightarrow -y$, and $x \leftrightarrow y$, let $\mathbf{p}=(p_1,p_2)^T$ with $p_1\geq p_2 \geq 0$ and $p_1>0$. 

The Jacobian of the Galerkin truncated vector field \eqref{eq:fourierodes} is
\begin{equation}
	J_{\mathbf{a},\mathbf{b}}=
	\begin{cases}
		0\;\;\;\text{if }\mathbf{a}=\mathbf{b}\text{ or }\mathbf{b}=\mathbf{0}\\
		\left ( \frac{1}{|\mathbf{b}|^2}  -\frac{1}{|{\mathbf{b}-\mathbf{a}}|^2}  \right )\mathbf{b} \times \mathbf{a} \;\omega_{{\mathbf{a}-\mathbf{b}}}.
	\end{cases}
\end{equation}

The Jacobian of the Zeitlin truncated vector field \eqref{eq:zeitlinode} is 
\begin{equation}
	J'_{\mathbf{a},\mathbf{b}}=
	\begin{cases}
		0\;\;\;\text{if }\mathbf{a}=\mathbf{b}\text{ or }\mathbf{b}=\mathbf{0}\\
		\frac{1}{\varepsilon}\left ( \frac{1}{|\mathbf{b}|^2} \sin (\varepsilon \mathbf{b} \times \mathbf{a})
			+\frac{1}{|\widehat{(\mathbf{b}-\mathbf{a})}|^2}\sin (\varepsilon \mathbf{a} \times \widehat{(\mathbf{b}-\mathbf{a})}) \right ) \omega_{\widehat{(\mathbf{a}-\mathbf{b})}}.
	\end{cases}
\end{equation}

Evaluating these at the equilibrium \eqref{eq:fourierequil} gives the linearised systems
\begin{equation}
\label{eq:fullDEG}
\dot{\omega}_{\mathbf{k}}=
	\Gamma \left ( \frac{1}{|\mathbf{p}|^2}-\frac{1}{|\mathbf{k}+\mathbf{p}|^2}\right )\mathbf{k}\times\mathbf{p} \; \omega_{\mathbf{k}+\mathbf{p}}-
	\Gamma \left ( \frac{1}{|\mathbf{p}|^2}-\frac{1}{|\mathbf{k}-\mathbf{p}|^2}\right )\mathbf{k}\times\mathbf{p} \; \omega_{\mathbf{k}-\mathbf{p}}
\end{equation}
for the Galerkin truncation and 
\begin{equation}
\label{eq:fullDEZ}
\begin{aligned}
\dot{\omega}_{\mathbf{k}}=
\frac{\Gamma}{\varepsilon}\left ( \frac{1}{|\widehat{\mathbf{k}-\mathbf{p}}|^2}\sin (\varepsilon (\widehat{\mathbf{k}-\mathbf{p}})	\times \mathbf{k})+\frac{1}{|\mathbf{p}|^2}\sin(\varepsilon \mathbf{p} 	\times \mathbf{k}) \right ) \omega_{\widehat{\mathbf{k}-\mathbf{p}}} \\
% \notag
+\frac{\Gamma}{2	\varepsilon}\left ( \frac{1}{|\widehat{\mathbf{k}+\mathbf{p}}|^2}\sin (\varepsilon (\widehat{\mathbf{k}+\mathbf{p}})	\times \mathbf{k})+\frac{1}{|\mathbf{p}|^2}\sin(\varepsilon \mathbf{k} 	\times \mathbf{p}) \right ) \omega_{\widehat{\mathbf{k}+\mathbf{p}}}\,
\end{aligned}
\end{equation}
for the Zeitlin truncation. 

\subsection{Decoupling into Classes}
\label{sec:Decoupling}

The key observation is that $\dot{\omega}_{\mathbf{k}}$ depends only ${\omega}_{\mathbf{k}\pm \mathbf{p}}$. Thus the linearised systems can be block-diagonalised. This block-diagonalisation
is analogous to the construction in \cite{Li00}. Following Li we call the individual blocks classes, 
which leads to the following definition:
\begin{definition}[Classes]
For some $\mathbf{a}\in\mathcal{D}$ (and $\mathbf{p}$ fixed by the choice of equilibrium), the class $\Sigma_\mathbf{a} \subset \mathcal{D}$ is defined for the Galerkin truncation by 
\begin{equation}
	\Sigma_\mathbf{a} = \{ \mathbf{a}+k\mathbf{p}\in \mathcal{D} \; | \; k \in \mathds{Z} \}.
\end{equation}
 or equivalently for the Zeitlin truncation
\begin{equation}
	\Sigma'_\mathbf{a} = \{ \widehat{\mathbf{a}+k\mathbf{p}}\in \mathcal{D} \; | \; k \in \mathds{Z} \}.
\end{equation}
\end{definition}
\begin{figure}
% g54r32
\centering
%
%     \begin{tikzpicture}[scale=0.9]
% 
% \draw [name path=mycircle, fill, fill opacity=0.1] (0,0) circle [radius=sqrt(10)];
% 
% 
% 
% 
% \draw [name path=circup,domain=atan(0.333333)+60:atan(0.333333)-60 , opacity=0] plot ({sqrt(9.9)*cos(\x)}, {sqrt(9.9)*sin(\x)});
% 
% \draw [name path=curveup,domain=atan(0.333333)+120:atan(0.333333)+240, dashed] plot ({sqrt(10)*cos(\x)+3}, {sqrt(10)*sin(\x)+1});
% 
% \draw [name path=circdown,domain=atan(0.333333)+240:atan(0.333333)+120, opacity=0] plot ({sqrt(9.9)*cos(\x)}, {sqrt(9.9)*sin(\x)});
% 
% \draw [name path=curvedown,domain=atan(0.333333)-60:atan(0.333333)+60, dashed]plot ({sqrt(10)*cos(\x)-3}, {sqrt(10)*sin(\x)-1});
% 
% 
% 
%   \tikzfillbetween[
%     of=curveup and circup ] {fill=white, fill opacity=1};
% 
% 
%   \tikzfillbetween[
%     of=curvedown and circdown] {fill=white, fill opacity=1};
%     
% 
%     
% \foreach \x in {-4,...,4}
% 	\foreach \y in {-4,...,4}
% 		\draw [fill] (\x,\y) circle [radius=.01];
%     
% \draw [<->] (-4,-3.666667) -- (4.3,-0.9);
% 
% \draw [fill, red] (-2,-3) circle [radius=0.07];
% \draw [fill, red] (1,-2) circle [radius=0.07];
% \draw [fill, red] (4,-1) circle [radius=0.07];
% 
% 
% \draw [<->] (-4,-0.33333) -- (4,2.333333);
% 
% 
% \node [right] at (4,2.333333) {$A$};
% \node [right] at (4,-0.8) {$B$};
% 
% \draw [fill, red] (-3,0) circle [radius=0.07];
% \draw [fill, red] (0,1) circle [radius=0.07];
% \draw [fill, red] (3,2) circle [radius=0.07];
% 
% \draw [->] (0,0) -- (3,1);
% \node [above right] at (3,1) {$\mathbf{p}=(3,1)$};
% 
% \node [below right] at (0,0) {$(0,0)$};
% 
% \end{tikzpicture}

\includegraphics{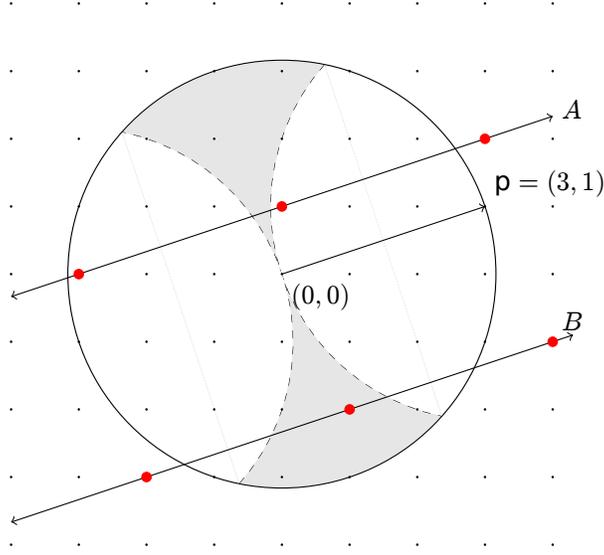}

\caption{The differential equations governing the set of Fourier Coefficients %($\subset \mathds{Z}^2$) 
decouple into `classes' when linearised. For $\mathbf{a}\in \mathds{Z}^2$, $\dot{\omega}_{\mathbf{a}}$ depends only on $\omega_{\mathbf{a}+\mathbf{p}}$ and $\omega_{\mathbf{a}-\mathbf{p}}$. Extending this idea we get a subset of coefficients that only depend on each other, \emph{the class led by }$\mathbf{a}$. These coefficients all lie on a straight line with direction $\mathbf{p}$. 
Classes that do not have an intersection with the disc indicated are stable. Clases that intersect the shaded region indicated at a lattice point have a pair of real eigenvalues (Theorem \ref{thm:realevs}). Classes that intersect the disc but do not have intersect the shaded region at a lattice point lead to either a complex quadruplet or two pairs of real eigenvalues.}
\label{fig:classes}
\end{figure}
Figure \ref{fig:classes} illustrates this idea. Note that the Zeitlin truncated classes $\Sigma'_\mathbf{a}$ `wrap around' the domain $\mathcal{D}$. 

\cite{Li00} makes the analogous definition for the non-truncated system. In that paper, the classes are infinitely large, and there are infinitely many classes. 
In our definition, there are finitely many classes of finite size, which depend on the truncation size $N$. As $\mathcal{D}$ is finite, $\Sigma_\mathbf{a}$ and $\Sigma'_\mathbf{a}$ are both finite. 
Write $\mathbf{p}=( p_1 , p_2 )^T$, and $\kappa=\text{gcd}(p_1,p_2)$. Then

\begin{equation}
\label{eq:ClassSizeG}
\left|\Sigma_\mathbf{a}\right| \leq \left \lfloor \frac{2N+1}{\text{max}(p_1,p_2)} \right \rfloor
\end{equation} 

and 
\begin{equation}
\label{eq:ClassSizeZ}
\left|\Sigma'_\mathbf{a}\right|=\frac{2N+1}{\text{gcd}(2N+1,\kappa)}.
\end{equation}
Note $\left|\Sigma'_\mathbf{a}\right|$ does not depend on $\mathbf{a}$, and is odd for all choices of $N$ and $\mathbf{p}$. We use this fact many times later. For $\text{gcd}(p_1,p_2)=1$, $|\Sigma'_\mathbf{a}|=2N+1$.

For $\Sigma_\mathbf{a}$ we can make a ``canonical'' choice of $\mathbf{a}$ by selecting $\mathbf{a}$ in a length $|\mathbf{p}|$ strip. For $\Sigma'_\mathbf{a}$ a canonical choice for $\mathbf{a}$ is found by restricting $\mathbf{a}$ to be in a $|\mathbf{p}|$ by $\frac{|\Sigma'_{\mathbf{a}}|}{|\mathbf{p}|}$ rectangle centred around $\mathbf{0}$ oriented so the sides of length $|\mathbf{p}|$ are parallel to $\mathbf{p}$. 
%This is illustrated in figure \ref{fig:canchoice}. 
For the Zeitlin truncation a unique choice is not possible if $\kappa>1$. See Section \ref{sec:Classes} and \ref{lem:NSequence} for details.

Fixing $\mathbf{p}$ and a canonical choice of $\mathbf{a}$ we now restrict our attention to the associated subsystems $\Sigma_\mathbf{a}$ and $\Sigma'_\mathbf{a}$. Introduce new notation
\begin{equation}
\omega_k=\begin{cases} \omega_{\mathbf{a}+k\mathbf{p}}& \text{ if } \mathbf{a}+k\mathbf{p}{\in}\mathcal{D}, \\ 0&\text{ otherwise.}\end{cases}
\;\;\;\;\;\;\;\;\;\;\;\;
\omega'_k=\omega_{\widehat{\mathbf{a}+k\mathbf{p}}},
\end{equation}
\begin{equation}
	\label{eq:rhodef}
	\rho_k=\frac{1}{|\mathbf{p}|^2}-\frac{1}{|\widehat{\mathbf{a}+k\mathbf{p}}|^2}\;\; \text{for} \;\;k=\in\mathbb{Z},
\end{equation}
\begin{equation}
	\alpha=\Gamma \mathbf{a}\times \mathbf{p} \in \mathds{R}, \;\;\;
	\alpha'=\frac{\Gamma\sin (\varepsilon (\widehat{\mathbf{a}+k\mathbf{p}}) 	
	    \times \mathbf{p})}{\varepsilon} \in \mathds{R}.
\end{equation}
The value of the coefficient $\rho_k$ is related to the distance of the lattice point 
$\widehat{ \mathbf{a} + k \mathbf{p}}$ from the boundary of the disc of radius $|\mathbf{p}|$, 
where negative values of $\rho_k$ correspond to lattice points inside the disc.
Note that $\alpha$ does not depend on $k$. 
%as $(\mathbf{a}+k\mathbf{p})	\times \mathbf{p}=\mathbf{a}	\times\mathbf{p}+k\mathbf{p}	\times\mathbf{p}=\mathbf{a}	\times\mathbf{p}$. 
Also note that
\begin{equation}
\label{eq:AlphaLimit}
\lim_{N\to \infty} \alpha'=\lim_{\varepsilon \to 0}\alpha' = \Gamma \mathbf{a}	\times \mathbf{p}=\alpha.
\end{equation}

Noting that $\sin( \epsilon ( \widehat{ \kk + \pp}) \times \kk) = \sin ( \epsilon \pp \times \kk)$
and rewriting \eqref{eq:fullDEG} and \eqref{eq:fullDEZ} with the new notation gives a compact form of the linear systems
\begin{align}
\label{eq:classDEG}
&\dot{\omega}_k=\alpha ( \rho_{k+1}\omega_{k+1}-\rho_{k-1}\omega_{k-1}), \\
\label{eq:classDEZ}
&\dot{\omega}'_k=\alpha' ( \rho_{k+1}\omega'_{k+1}-\rho_{k-1}\omega'_{k-1}).
\end{align}
All indices in \eqref{eq:classDEZ} are written modulo $|\Sigma'_\mathbf{a}|$.

Let $m_1,m_2\in\mathbb{Z}$ be such that $\mathbf{a}-m_1\mathbf{p}\in\mathcal{D}$, $\mathbf{a}-(m_1+1)\mathbf{p}\not{\in}\mathcal{D}$ and $\mathbf{a}+m_2\mathbf{p}\in\mathcal{D}$, $\mathbf{a}+(m_2+1)\mathbf{p}\not{\in}\mathcal{D}$. If we write $\boldsymbol{\omega}=\begin{pmatrix}\omega_{-m_1},\;\omega_{-m_1+1},\;...,\omega_{-1},\omega_0,\omega_{+1},\;...,\omega_{m_2}\end{pmatrix}^T$, then $\boldsymbol{\dot{\omega}}=\alpha A\boldsymbol{\omega}$ where 
\begin{equation}
\label{eq:MainMatrixG}
A= \begin{pmatrix}
		0 & +\rho_{-m_1+1} & 0  & \cdots & 0 & 0 & 0 \\
		-\rho_{-m_1} & 0 & +\rho_{-m_1+2} &  \cdots  & 0 & 0 &  0 \\
		0 & -\rho_{-m_1+1} & 0 &  \cdots & 0 & 0 & 0 \\
		\vdots & \vdots & \vdots &  \ddots & \vdots & \vdots & \vdots \\
		0 & 0 & 0 &  \cdots & 0 & +\rho_{m_2-1} & 0  \\
		0 & 0 & 0 &  \cdots & -\rho_{m_2-2} & 0 & +\rho_{m_2}  \\
		0 & 0 & 0 &  \cdots & 0 & -\rho_{m_2-1} & 0
	\end{pmatrix}.
\end{equation} 

If we write $\boldsymbol{\omega}'=\begin{pmatrix}\omega'_0,\;\omega'_1,\;...,\;\omega'_{n-1}\end{pmatrix}^T$, then $\boldsymbol{\dot{\omega}}'=\alpha' A'\boldsymbol{\omega}'$ where 
\begin{equation}
\label{eq:MainMatrixZ}
A'= \begin{pmatrix}
		0 & +\rho_1 & 0 & 0 & \cdots & 0 & -\rho_{n-1} \\
		-\rho_0 & 0 & +\rho_2 & 0 & \cdots & 0 &  0 \\
		0 & -\rho_1 & 0 & +\rho_3 & \cdots & 0 & 0 \\
		0 & 0 & -\rho_2 & 0 &  \cdots & 0 & 0 \\
		\vdots & \vdots & \vdots & \vdots & \ddots & \vdots & \vdots \\
		0 & 0 & 0 & 0 & \cdots & 0 & +\rho_{n-1}  \\
		+\rho_0 & 0 & 0 & 0 & \cdots & -\rho_{n-2} & 0
	\end{pmatrix}.
\end{equation} 

If $\mathbf{a}=\mathbf{0}$ or $\mathbf{a}$ is parallel to $\mathbf{p}$, $\alpha=\alpha'=0$. Thus the associated class only contributes zero eigenvalues and will not contribute to the linear instability of the system. We can thus ignore the classes with $\alpha=\alpha'=0$.
%However, it does not satisfy the properties described in Section \ref{sec:zeitlintruncation}.
% see comment below

Note that $A$ can be written as $A=JS$ where
\begin{equation}
\label{eq:AntisymmetricG}
J= \begin{pmatrix}
		0 & +1 & 0 & \cdots & 0 & 0 \\
		-1 & 0 & +1 &  \cdots & 0 &  0 \\
		0 & -1 & 0 &  \cdots & 0 & 0 \\
		\vdots & \vdots &  \vdots & \ddots & \vdots & \vdots \\
		0 & 0 & 0 &  \cdots & 0 & +1  \\
		0 & 0 & 0  & \cdots & -1 & 0
	\end{pmatrix},
	\;\;
S= \begin{pmatrix}
		\rho_{-m_1} & 0 & 0 & \cdots &  0 \\
		0 & \rho_{-m_1+1} & 0 &  \cdots  &  0 \\
		0 & 0 & \rho_{-m_1+2} &  \cdots  & 0 \\
		\vdots & \vdots &  \vdots & \ddots &  \vdots \\
		0 & 0 & 0  & \cdots & \rho_{m_2}
	\end{pmatrix}.
\end{equation} 
%
%Also, $A'$ can be written as $A'=J'S'$ where
%\begin{equation}
%\label{eq:AntisymmetricZ}
%J'= \begin{pmatrix}
%		0 & +1 & 0 & \cdots & 0 & -1 \\
%		-1 & 0 & +1 &  \cdots & 0 &  0 \\
%		0 & -1 & 0 &  \cdots & 0 & 0 \\
%		\vdots & \vdots &  \vdots & \ddots & \vdots & \vdots \\
%		0 & 0 & 0 &  \cdots & 0 & +1  \\
%		+1 & 0 & 0  & \cdots & -1 & 0
%	\end{pmatrix},
%	\;\;
%S'= \begin{pmatrix}
%		\rho_0 & 0 & 0 & \cdots &  0 \\
%		0 & \rho_1 & 0 &  \cdots  &  0 \\
%		0 & 0 & \rho_2 &  \cdots  & 0 \\
%		\vdots & \vdots &  \vdots & \ddots &  \vdots \\
%		0 & 0 & 0  & \cdots & \rho_{n-1}
%	\end{pmatrix}.
%\end{equation} 

A similar construction can be made for $A'=J'S'$. As $J$ and $J'$ are skew-symmetric and $S$ and $S'$ are symmetric, these are (non-canonical) Hamiltonian systems. For both systems, $\mathcal{H}=\sum_k \rho_k\omega_k^2 $.\footnote{As $J'$ is circulant, one could write down its eigensystem explicitly by applying a discrete Fourier transform (see \cite{Karner02}) and therefore find a set of canonical coordinates (see \cite{Silva01}). This will not be used in this paper, but may be of interest.} From this it follows that if $\lambda$ is an eigenvalue of $A$ or $A'$ then  $-\lambda$, $\bar{\lambda}$ and $-\bar{\lambda}$ are also eigenvalues. Note that $\text{det}(J')=0$ as $J'$ has odd size, and therefore $J'$ is not symplectic with a one-dimensional kernel. $J$ is symplectic if and only if $|\Sigma_{\mathbf{a}}|$ is even.

%We can also write $B = \tilde J S$ with $S$ as in \eqref{eqn:Antisymmetric} and $\tilde{J}$ like  $J$ in \eqref{eq:Antisymmetric} but without the top right and bottom left elements of $J$.
%But what of it? I checked, dim(ker(\tilde(J)))=1 for odd sizes, 0 for even sizes...  
%Since $|\Sigma_\mathbf{a}|$ is always odd, this again has one-dimensional 
%kernel like $J$. So again we have a Hamiltonian system, however, with a 
%different Hamiltonian structure. Only the structure $J$ is derived from the truncated
%Poisson structure $J_{\mathbf{k, l}}$, while $\tilde J$ appears `by accident' after 
%linearising the non-canonical Galerkin truncation.

We now focus on the behaviour of the eigenvalues of $A$ and $A'$ as a function of the $\rho_k$ values. Note that there is a symmetry in $\mathbf{a}$. For every class $\Sigma_\mathbf{a}$ or $\Sigma'_\mathbf{a}$, the class $\Sigma_{-\mathbf{a}}$ or $\Sigma'_{-\mathbf{a}}$ generates the same set of eigenvalues. It is worth noting that $\alpha(-\mathbf{a})=-\alpha(\mathbf{a})$ and $\alpha'(-\mathbf{a})=-\alpha'(\mathbf{a})$, but as all eigenvalues occur in $\pm$ pairs, this does not affect the spectrum. Thus for the full system rather than a particular class all eigenvalues occur with even multiplicity.

\begin{definition}[The Unstable Disc] Introduce the disc $D_\mathbf{p}$ 
\begin{equation}
	D_\mathbf{p}=\{ \mathbf{x}\in \mathcal{D} \; | \; |\mathbf{x}|<|\mathbf{p}|\}.
\end{equation}
\end{definition}
This disc is shown in figure \ref{fig:classes}.  A simple but important observation is
\begin{lemma}
A lattice point is inside the unstable disc if and only if the corresponding $\rho$ is negative:
\begin{equation}
\label{eq:insidenegative}
	\widehat{\mathbf{a}+k\mathbf{p}} \in D_\mathbf{p} \; \; \iff \; \; \rho_k<0.
\end{equation}\end{lemma}
\begin{proof} This is true as $\rho_k<0$ if and only if $|\mathbf{a}+k\mathbf{p}|<|\mathbf{p}|$ from \eqref{eq:rhodef}, which is exactly the condition that $\widehat{\mathbf{a}+k\mathbf{p}}\in D_\mathbf{p}$. 
\end{proof}

This is illustrated in figure \ref{fig:classes}. The point $\mathbf{a}$ inside the disc corresponds to $\rho_0<0$ and the other points correspond to $\rho_k>0$. Also note that 
$$\widehat{\mathbf{a}+k\mathbf{p}} \in \partial D_\mathbf{p} \iff \rho_k=0.$$

\section{Stability and Instability of Classes}
\label{sec:Instability}

\subsection{Stable Classes}

The matrices $A$ and $A'$ are similar to skew-symmetric matrices by conjugation.

\begin{equation}
T=\begin{pmatrix}
	\sqrt{\rho_{-m_1}} & 0 & 0 & \cdots & 0 \\
	0 & \sqrt{\rho_{-m_1+1}} & 0 & \cdots & 0 \\
	0 & 0 & \sqrt{\rho_{-m_1+2}} & \cdots & 0 \\
	\vdots & \vdots & \vdots & \ddots & \vdots  \\
	0 & 0 & 0 & \cdots & \sqrt{\rho_{m_2} }
\end{pmatrix},
\end{equation}
\begin{equation}
\label{eq:RealSkewSymmetricG}
 TAT^{-1}= \begin{pmatrix}
		0 & +\sqrt{\rho_{-m_1}\rho_{-m_1+1}} & 0 & \cdots & 0 \\
		-\sqrt{\rho_{-m_1}\rho_{-m_1+1}} & 0 & +\sqrt{\rho_{-m_1+1}\rho_{-m_1+2}} & \cdots &  0 \\
		0 & -\sqrt{\rho_{-m_1+1}\rho_{-m_1+2}} & 0  & \cdots & 0 \\
		\vdots & \vdots & \vdots & \ddots &  \vdots \\
		0  & 0 & 0 & \cdots  & 0
	\end{pmatrix}.
\end{equation}
%\begin{equation}
%T'=\begin{pmatrix}
%	\sqrt{\rho_0} & 0 & 0 & \cdots & 0 \\
%	0 & \sqrt{\rho_1} & 0 & \cdots & 0 \\
%	0 & 0 & \sqrt{\rho_2} & \cdots & 0 \\
%	\vdots & \vdots & \vdots & \ddots & \vdots  \\
%	0 & 0 & 0 & \cdots & \sqrt{\rho_{n-1} }
%\end{pmatrix},
%\end{equation}
%\begin{equation}
%\label{eq:RealSkewSymmetricZ}
% T'A'T'^{-1}= \begin{pmatrix}
%		0 & +\sqrt{\rho_0\rho_1} & 0 & \cdots & -\sqrt{\rho_{n-1}\rho_0} \\
%		-\sqrt{\rho_0\rho_1} & 0 & +\sqrt{\rho_1\rho_2} & \cdots &  0 \\
%		0 & -\sqrt{\rho_1\rho_2} & 0  & \cdots & 0 \\
%		\vdots & \vdots & \vdots & \ddots &  \vdots \\
%		+ \sqrt{\rho_{n-1}\rho_0}  & 0 & 0 & \cdots  & 0
%	\end{pmatrix}.
%\end{equation}
A very similar construction exists for $A'$.

If $\rho_k> 0$ for all $k$, this transformation is real and thus $A$ and $A'$ are similar to real skew-symmetric matrices. Thus all eigenvalues are purely imaginary and $A$ and $A'$ can be diagonalised and so the class is linearly stable. By \eqref{eq:insidenegative} this condition is true exactly if 
\begin{equation} \label{eqn:stablecond}
	\Sigma_{\mathbf{a}}\cap D_\mathbf{p}=\emptyset.
\end{equation}

This is the finite-dimensional analogue of Li's \emph{Unstable Disc Theorem} (Theorem III.1) in \cite{Li00}, though the method of proof used in that paper is naturally very different. A discussion of the details such as choice of $\mathbf{p}$ and $N$ required for \eqref{eqn:stablecond} to hold follows in Section \ref{sec:Classes}.
Because of this result, only classes with $\mathbf{a}\in D_\mathbf{p}$ can contribute linear instability. Also $\mathbf{a}=\mathbf{0}$ implies $\alpha=\alpha'=0$ and so this class cannot contribute linear instability. 

%HRD: Li proves Lyapunov stability but 
%excludes lattice points on the boundary. We allow lattice points on the boundary. 
%Can we take the limit $N\to \infty$ to get linear stability of the PDE? If so, 
%we may want to state this as a theorem. Or can the additional zero eigenvalues
%cause a problem in the limit? 
%JW: Li proves Lyapunov stability in the boundary lattice points also, his "half class stability" theorem

%%!!!!!!!!!!!!!!!!!!!!!!!!!!!!!!!!!!!!!!!!!!!!!!!!!!!!!!!!!!!!!!!!!!!!%%
% Section on Krein Theory removed, for expansion elsewhere (Hopefully) %
%%!!!!!!!!!!!!!!!!!!!!!!!!!!!!!!!!!!!!!!!!!!!!!!!!!!!!!!!!!!!!!!!!!!!!%%

\subsection{Unstable Classes}
\label{sec:Unstable}

For classes with $\mathbf{a}\in D_\mathbf{p}$, there are two primary possibilities to consider:
\begin{enumerate}[i)]
\item There is exactly one intersection between the class and the disc (ie, $\Sigma_{\mathbf{a}}\cap D_\mathbf{p}=\{\mathbf{a}\}$). This can only occur when $\mathbf{a}$ is chosen to be in the shaded area indicated in figure \ref{fig:classes}.
\item There are exactly two consecutive intersections between the class and the disc (ie, \\ ${\Sigma_{\mathbf{a}}\cap D_\mathbf{p}=\{\mathbf{a}, \mathbf{a}+\mathbf{p}\}}$ or $\Sigma_{\mathbf{a}}\cap D_\mathbf{p}=\{\mathbf{a}, \mathbf{a}-\mathbf{p}\}$). This occur when $\ab\in D_\pp$ is chosen outside the shaded area indicated in figure \ref{fig:classes}.
\end{enumerate}

For the Zeitlin style truncation, there is also a third possibility we must consider:

\begin{enumerate}[iii)]
\item There are at least two non-consecutive intersections between the class and the disc (ie, $\mathbf{a}\in \Sigma'_{\mathbf{a}}\cap D_\mathbf{p}$ and $\widehat{\mathbf{a}+k\mathbf{p}}\in\Sigma_{\mathbf{a}}\cap D_\mathbf{p}$ for some $k\neq -1,0,1$).
\end{enumerate}
Note that points on the boundary are treated as being outside the disc. 
Also note that it is not possible for three consecutive lattice points in a class to be in the unstable disc. 
If $\mathbf{a}$, $\mathbf{a}-\mathbf{p}$ and $\mathbf{a}+\mathbf{p}$ were are all in $\overline{ D_\mathbf{p}}$, they would lie along a diameter as  $\overline{ D_\mathbf{p}}$ has diameter $2|\mathbf{p}|$ and the distance from $\mathbf{a}-\mathbf{p}$ to $\mathbf{a}+\mathbf{p}$ is $2|\mathbf{p}|$. 
Therefore $\ab = (0,0)$ and $\mathbf{a}\pm \mathbf{p} \in \partial  \overline{ D_\mathbf{p}}$.
This is the only possibility to have three consecutive lattice point in  $\overline{ D_\mathbf{p}}$, 
and hence $D_\pp$ can at most contain two consecutive lattice points.
Figure \ref{fig:classes} makes this idea clear.

From our numerical and analytical results we can categorise the spectrum of the class in these three cases:
\begin{enumerate}[i)]
\item The spectrum has a single pair of real eigenvalues and all other eigenvalues on the imaginary axis. This is proved in Theorem~\ref{thm:realevs}.
\item The spectrum typically corresponds to a quartet of complex eigenvalues $\pm\alpha\pm\beta i$, and all other eigenvalues on the imaginary axis. It  can also correspond to two pairs of real eigenvalues, though seems to be less common.
\item This corresponds to the class `wrapping around' the truncated domain of lattice points and intersecting the disc again; see \ref{fig:Problems}. The spectrum is a combination of case (i) and case (ii) according to how 
successive intersections with the disc occur.
\end{enumerate}
This last case is atypical and does not occur with the Galerkin truncation. Usually this case can be avoided by 
a proper choice of $N$, however when both entries of  $\mathbf{p}$ are even, it cannot be 
avoided.  This is discussed in detail in Section \ref{sec:Classes}, particularly Lemma \ref{lem:NSequence}.

%For instance, if the class passes through the disc with two interior points, then passes through the disc again with one interior point, we observe one complex quartet and one real pair of eigenvalues.

All our numerical evidence is consistent with the result in \cite{Li04} that the number of eigenvalues with non-zero real part is  $\leq 2|D_\mathbf{p}|$ (twice the number of interior lattice points in the unstable disc), and our observation is that this is the {\em exact} number of hyperbolic eigenvalues.

For case (i), Theorem~\ref{thm:realevs} proves that this case always leads to non-zero real eigenvalues. 
For some $\pp$ and $\ab$ leading to case (i), Theorem~\ref{thm:main} describes an explicit lower bound which is independent of the truncation size $N$ for the real eigenvalues. \hl{This agrees with numerically observed results, eg the} $\epsilon=0$ \hl{inviscid result shown in Figure 2 of} \cite{li08}.
It should be noted that our methods do not preclude the possibility that there are other eigenvalues with non-zero real part unaccounted for; we simply assert that there is at least one eigenvalue with positive real part. This together with the results from \cite{Friedlander97}, \cite{Li04}, and \cite{shvidkoy03} is sufficient to conclude nonlinear instability for the whole system.

For the following section, we consider a general set of parameters $(a_0,a_1,...,a_{n-1})$ instead of $\rho_k$. Introduce a tridiagonal matrix 
\begin{equation}
\label{eq:tmatrix}
T_\alpha^\beta=
\begin{pmatrix}
0 & a_{\alpha+1} & 0 & 0 & \cdots & 0  & 0 \\
-a_{\alpha} & 0 & a_{\alpha+2} & 0 & \cdots & 0 & 0 \\
0 & -a_{\alpha+1} & 0 & a_{\alpha+3} & \cdots & 0 & 0 \\ 
0 & 0 & -a_{\alpha+2} & 0 & \cdots & 0 & 0 \\
0 & 0 & 0 & -a_{\alpha+3} & \cdots & 0 & 0 \\
\vdots & \vdots & \vdots & \vdots & \ddots & \vdots & \vdots \\
0 & 0 & 0 & 0 & \cdots & 0 & a_{\beta} \\
0 & 0 & 0 & 0 & \cdots & -a_{\beta-1} & 0
\end{pmatrix},
\end{equation}
and its characteristic polynomial
\begin{equation}
\mathcal{T}_\alpha^\beta(x)=\text{det}(xI-T_\alpha^\beta)
\end{equation}
for some integers $0\leq \alpha < \beta \leq n-1$.
Then $\mathcal{T}_\alpha^\beta(x)$ can be recursively defined by expansion from top left to bottom right
\begin{align}
& \mathcal{T}_\alpha^\alpha(x)=1,\;\;\;\;\;\;\mathcal{T}_\alpha^{\alpha+1}(x)=x^2+ a_{\alpha+1}a_\alpha, \nonumber \\
&\mathcal{T}_\alpha^\beta(x)=x\mathcal{T}_\alpha^{\beta-1}(x)+a_\beta a_{\beta-1}\mathcal{T}_\alpha^{\beta-2}(x). \label{eq:recurrence1}
\end{align}
or by expansion from bottom right to top left
\begin{align}
& \mathcal{T}_\beta^\beta(x)=1,\;\;\;\;\;\;\mathcal{T}_{\beta-1}^{\beta}(x)=x^2+a_{\beta-1}a_\beta , \nonumber \\
&\mathcal{T}_\alpha^\beta(x)=x\mathcal{T}_{\alpha+1}^{\beta}(x)+a_{\alpha} a_{\alpha+1}\mathcal{T}_{\alpha+2}^{\beta}(x).
\label{eq:TRecurrence}
\end{align}
Note that
\begin{equation}
\label{eq:TPositive}
	a_k>0 \text{ for all } \alpha\leq k \leq \beta \; \implies \; \mathcal{T}_\alpha^\beta(x)>0 \text{ for all }x>0.
\end{equation}
This can be seen by the recursive definitions; all terms are positive. \footnote{Note that the \eqref{eq:recurrence1} and \eqref{eq:TRecurrence} satisfy the condition for Favard's theorem (see \cite{Favard35}) . Thus the polynomials $\mathcal{T}(x;a_1,...a_j)$ are orthogonal for $j=1,2,3,...$ with respect to an inner product with some weight function (see \cite{Szego39}). However, as the $\alpha_k$ terms may be negative this weight function will not always be positive. This will not be used here, but may be useful in future work for describing the imaginary part of the spectrum.}

%new
The following is also useful:
\begin{equation}\mathcal{T}_\alpha^\beta(0)=\begin{cases}
   \prod_{k=\alpha}^\beta a_k & \text{if } \beta-\alpha \text{ is odd,} \\
    0      & \text{if } \beta-\alpha \text{ is even.}
  \end{cases}\end{equation}
  \label{eq:tat0}
\begin{equation} \left . \frac{\mathrm{d}}{\mathrm{d}x}\mathcal{T}_\alpha^\beta(x)\right |_{x=0}=\begin{cases}
 0 & \text{if } \beta-\alpha \text{ is odd,} \\
  \sum_{k=0}^\frac{\beta-\alpha}{2} \left ( \prod_{j=\alpha;\;j\neq \alpha+2k}^\beta a_j \right )  & \text{if } \beta-\alpha \text{ is even.}
\label{eq:dat0}
  \end{cases}\end{equation}

These can be proved by simple induction arguments. %should I include this?

Introduce similar notation for \eqref{eq:MainMatrixZ}
\begin{equation}
\label{eq:GeneralMatrix}
A'= \begin{pmatrix}
		0 & +a_1 & 0 & 0 & \cdots & 0 & -a_{n-1} \\
		-a_0 & 0 & +a_2 & 0 & \cdots & 0 &  0 \\
		0 & -a_1 & 0 & +a_3 & \cdots & 0 & 0 \\
		0 & 0 & -a_2 & 0 &  \cdots & 0 & 0 \\
		\vdots & \vdots & \vdots & \vdots & \ddots & \vdots & \vdots \\
		0 & 0 & 0 & 0 & \cdots & 0 & +a_{n-1}  \\
		+a_0 & 0 & 0 & 0 & \cdots & -a_{n-2} & 0
	\end{pmatrix},
\end{equation}
with the characteristic polynomial
\begin{equation}
\label{eq:FullCharacteristic}
\mathcal{A}(x)=\text{det}(xI-A).
\end{equation}

Then, if $n$ is odd, 
\begin{align}
\mathcal{A}(x)&=x\mathcal{T}_0^{n-2}(x)+a_{n-1}a_{n-2}\mathcal{T}_0^{n-3}(x)+a_{n-1}a_0\mathcal{T}_1^{n-2}(x) \nonumber \\
&=\mathcal{T}_0^{n-1}(x)+a_0a_{n-1}\mathcal{T}_1^{n-2}(x) \label{eq:FullRecurrence}
\end{align}
This can be demonstrated by expanding by minors along the last row and column of $xI-A$. Recall that $n$ is odd for the relevant problem ($n=|\Sigma_{\mathbf{a}}|$ from equation \eqref{eq:ClassSizeZ}).

We first show that in case (i) there is some non-zero real eigenvalue. Because of the Hamiltonian nature of the system this means there is a plus/minus pair of eigenvalues and so there is linear instability. This is then extended to show that under certain conditions there is pair of real eigenvalues with an explicit lower bound independent of $N$.

\begin{theorem}[Real Eigenvalues in case (i)]
\label{thm:realevs}
If $\rho_0<0$, and $\rho_k\geq 0$ for all $k=1,2,...,n-1$ and $\rho_k=0$ for at most one of  $k=1,2,...,n-1$, then for sufficiently large $N$, \eqref{eq:MainMatrixZ} has a non-zero real eigenvalue.
\end{theorem}

\begin{proof}
The characteristic polynomial $\mathcal{A}$ for odd $n$ has leading term $x^n$ and constant term $0$. By combining \eqref{eq:FullRecurrence} and \eqref{eq:dat0} (noting that $n$ is odd, so $n-1$ and $n-3$ are even), the linear coefficient is given by
\begin{align}
\left . \frac{\mathrm{d}\mathcal{A}}{\mathrm{d}x}\right |_{x=0}& =\left . \frac{\mathrm{d}}{\mathrm{d}x}\mathcal{T}_0^{n-1}(x)\right |_0+a_0a_{n-1}\left . \mathcal{T}_1^{n-2}(x) \right |_0 \notag \\
&= \sum_{k=0}^\frac{n-1}{2} \left ( \prod_{j=0;\;j\neq 2k}^{n-1} a_j \right )+ \sum_{k=0}^\frac{n-3}{2} \left ( \prod_{j=1;\;j\neq 1+2k}^{n-2} a_j \right )  \notag \\
&= \sum_{k=0}^{n-1} \left ( \prod_{j=0;\;j\neq k}^{n-1} a_j \right )\label{eq:AlwaysValid} \\
&=\left (\prod_{j=0}^{n-1} a_j \right ) \left ( \sum_{k=0}^{n-1} \frac{1}{a_k} \right ). \label{eq:SometimesValid} 
\end{align}
Note that \eqref{eq:SometimesValid} is only valid for $a_k\neq 0$ for all $k$, but \eqref{eq:AlwaysValid} is always valid \footnote{The expression in equation \eqref{eq:AlwaysValid} is the so-called $n$-$1^{\textrm{st}}$ {\em elementary symmetric polynomial} in the variables $a_j$} . Now let $a_k=\rho_k$.
First assume $\rho_k>0$ for all $k\neq 0$. As $\rho_k\to \frac{1}{|\mathbf{p}|^2}>0$ as $|\widehat{\mathbf{a}+k\mathbf{p}}|\to \infty$ and the size $n$ of the classes grows linearly with $N$, $ \sum_{j=0}^{n-1} \frac{1}{a_j}=\sum_{j=0}^{n-1} \frac{1}{\rho_j}>0$ for sufficiently large $N$ (where $\rho_k=a_{k}$). However $\prod_{j=0}^{n-1} a_j <0$ as $a_0<0$, $a_j>0$ for all $j=1,..,n-1$, and hence the linear coefficient of the characteristic polynomial is less than zero.

If $\rho_k=0$ for exactly one $k$, then \eqref{eq:AlwaysValid} consists of only one term, $\prod_{j=0;\;j\neq k}^{n-1} \rho_j$. This is less than zero as $\rho_0<0$ and $\rho_j>0$ for all $j\neq 0,k$. 

If $\rho_k=0$ for more than one value of $k$, then the linear term is zero and we cannot apply this argument.

As the constant term is zero, and the linear term is non-zero, then the lowest order non-zero coefficient of the polynomial is negative.

We now argue by contradiction.
Assume all roots of the polynomial are imaginary (say $i\omega_k$) or complex ($\gamma_j+i\delta_j$) or zero. Then because eigenvalues occur in positive and negative pairs as well as conjugate pairs the polynomial has the form
\begin{align}\mathcal{A}&(x)= \\
&x^{n_1}\prod_{k=1}^{n_2}(x-i\omega_k)(x+i\omega_k)\prod_{j=1}^{n_3}\left [ (x-\gamma_j-i\delta_j)(x-\gamma_j+i\delta_j)(x+\gamma_j-i\delta_j)(x+\gamma_j+i\delta_j)\right ]  \nonumber \\
&=x^{n_1}\prod_{k=1}^{n_2}(x^2+\omega_k^2)\prod_{j=1}^{n_3}\left (x^4-2x^2(\gamma_j^2-\delta_j^2)+(\gamma_j^2+\delta_j^2)^2\right ). 
\end{align}
The lowest order non-zero coefficient (of $x^{n_1}$) is
$\prod_{k=1}^{n_2}(\omega_k^2)\prod_{j=1}^{n_3}\left ((\gamma_j^2+\delta_j^2)^2\right )>0$. Thus by contradiction there must be some real eigenvalue, which will occur in a plus and minus pair.
 
\end{proof}

For the Galerkin case \eqref{eq:tmatrix}, although the result is the same we cannot apply the same proof.
Since we were not able to show that these eigenvalues remain bounded away from the imaginary axis when 
$N \to \infty$ we now show in a different way that for both truncations there are eigenvalues whose 
real part is bounded away from zero for $N \to \infty$. These proofs have stricter conditions on the $\rho_k$
then Theorem~\ref{thm:realevs}, but we will see that they can be met for most $\mathbf{p}$.

\begin{lemma}[Lower Bound for Real Eigenvalue (Zeitlin)]
\label{thm:LowerBoundZ}
If $a_0<0$, and $a_k> 0 $ for all $k\neq 0$, and $a_0+a_2<0$, then
\begin{equation}
	\mathcal{A}(\sqrt{-a_1(a_0+a_2)})< 0.
\end{equation}
\end{lemma}
\begin{proof}
By expanding \eqref{eq:FullRecurrence} using \eqref{eq:TRecurrence},
\begin{align}
\mathcal{A}(x)= & (x^3+(a_0a_1+a_1a_2)x)\mathcal{T}_3^{n-1}(x)  \\
& +(x^2+a_0a_1)a_2a_3\mathcal{T}_4^{n-1}(x) \notag \\
& +a_0a_{n-1}\mathcal{T}_1^{n-2}(x). \notag
\end{align}
Thus 
\begin{align}
\mathcal{A}(\sqrt{-a_1(a_0+a_2)})= & -a_1a_2^2a_3\mathcal{T}_4^{n-1}(\sqrt{-a_1(a_0+a_2)}) \ \\
& +a_0a_{n-1}\mathcal{T}_1^{n-2}(\sqrt{-a_2(a_1+a_3)}). \notag
\end{align}
As $a_1,a_2,a_3,a_{n-1}> 0$, $a_0<0$ and the $\mathcal{T}$ terms are positive by \eqref{eq:TPositive} the result follows. 
\end{proof}

We make a similar case for the Galerkin truncation
\begin{lemma}[Lower Bound for Real Eigenvalue (Galerkin)]
\label{thm:LowerBoundG}
If $\beta>0$, $a_0<0$, and $a_k> 0 $ for all $k\neq 0$, and $a_0+a_2<0$, then
\begin{equation}
	\mathcal{T}_\alpha^\beta(\sqrt{-a_1(a_0+a_2)})< 0.
\end{equation}
\end{lemma}
\begin{proof}
Begin by noting that 
\begin{equation}
	\mathcal{T}_\alpha^\gamma(\sqrt{-a_1(a_0+a_2)})\geq 0
\end{equation} 
for all $\gamma<0$ by \eqref{eq:TPositive}.
By expanding \eqref{eq:TRecurrence},
\begin{align}
\mathcal{T}_\alpha^1(x)&=  (x^2+a_1a_0)\mathcal{T}_\alpha^{-1}(x) +a_0a_{-1}x\mathcal{T}_\alpha^{-2}(x) \\
	\mathcal{T}_\alpha^1(\sqrt{-a_1(a_0+a_2)})&=-a_1a_2\mathcal{T}_\alpha^{-1}(\sqrt{-a_1(a_0+a_2)})\\
	 & +a_0a_{-1}\sqrt{-a_1(a_0+a_2)}\mathcal{T}_\alpha^{-2}(\sqrt{-a_1(a_0+a_2)}). \notag
\end{align}
As $a_0<0$, $a_k>0$ for all $k\neq 0$, and $\mathcal{T}_\alpha^{-1},\mathcal{T}_\alpha^{-2}$ take positive values for positive arguments, $\mathcal{T}_\alpha^1(\sqrt{-a_1(a_0+a_2)})<0$.

Now similar to the proof of \ref{thm:LowerBoundZ},
\begin{align}
\mathcal{T}_\alpha^{2}(\sqrt{-a_1(a_0+a_2)})= & (x^2+a_1(a_0+a_2))x\mathcal{T}_\alpha^{-1}(\sqrt{-a_1(a_0+a_2)}) \\
 & +(x^2+a_2a_1)a_0a_{-1}\mathcal{T}_\alpha^{-2}(\sqrt{-a_1(a_0+a_2)}) \notag \\
& = -a_1a_0^2a_{-1}\mathcal{T}_\alpha^{-2}(\sqrt{-a_1(a_0+a_2)}). \notag
\end{align}
As $a_1,a_{-1}> 0$, $a_0<0$ and the $\mathcal{T}$ terms are positive by \eqref{eq:TPositive} it follows that $$\mathcal{T}_\alpha^{2}(\sqrt{-a_1(a_0+a_2)})<0.$$ 

Now if $\gamma>2$, we can make a recursive argument:
\begin{equation}
\mathcal{T}_\alpha^{\gamma}(x)=x\mathcal{T}_\alpha^{\gamma-1}(x)+a_\gamma a_{\gamma-1}\mathcal{T}_\alpha^{\gamma-2}(x).
\end{equation}
Now $x=\sqrt{-a_1(a_0+a_2)}>0$ and $a_\gamma a_{\gamma-1}>0$ (as $\gamma>2$). By inductive reasoning, as $\mathcal{T}_\alpha^{1}<0,$ $\mathcal{T}_\alpha^{2}<0$, then $\mathcal{T}_\alpha^{\gamma}<0$ for all $\gamma>2$. Therefore $T_\alpha^\beta(\sqrt{-a_1(a_0+a_2)})<0$.
\end{proof}

We can also make a similar construction if $a_0<0$, $a_k\geq 0$ for all $k\neq 0$ and $a_0+a_{n-2}<0$. In this case, $\mathcal{A}(\sqrt{-a_{n-1}(a_0+a_{n-2})})\leq 0$. 

\begin{lemma}
Assuming the same conditions as Lemma \ref{thm:LowerBoundZ} and Lemma \ref{thm:LowerBoundG}, there exists some $$x_1^*,x_2^*>\sqrt{-a_1(a_0+a_2)}$$ such that $\mathcal{T}_\alpha^\beta(x_1^*)=0$, $\mathcal{A}(x_2^*)=0$.
\end{lemma}
\begin{proof}
The leading order term of $\mathcal{A}(x)$ is always $x^n$ regardless of $a_0,...,a_{n-1}$, and $\mathcal{A}(x)$ is real for real $x$.  Thus $\lim_{x\to\infty}\mathcal{A}(x)>0$. But by Lemma \ref{thm:LowerBoundZ} $\mathcal{A}(\sqrt{-a_1(a_0+a_2)})<0$, and the result follows by the intermediate value theorem. 

The same argument  can be applied to $\mathcal{T}_\alpha^\beta(x)$, with the equivalent result.   
\end{proof}

%Note that it is possible that $\sqrt{-a_1(a_0+a_2)}=0$, and thus $x_i^*$ may also equal $0$.

We now turn our attention back to the context of our problem.

%\footnote{HRD: the theorem seems to have a problem when lattice points are on the boundary of 
%the disc: In Lemma 3 we need $<0$, in Lemma 2 we only show $\le$.
%But there are two choices, going up or down, one of them works, since 
%otherwise both $\rho_1 = \rho_{n-1} = 0$, which would imply that both 
%neighbouring lattice points are in the boundary of the disk. But this implies 
%$\alpha = 0$, the trivial case. 
%If there is one lattice point on the boundary, then we can hope that going 
%in the other direction works, but the condition $a_0 + a_{...} < 0$ still needs 
%to be checked, but it will hold for sufficiently large $\mathbf{p}$. 
%}
%JW: this is correct, and means that all the inequalities are non-strict.
% However, in the end we don't need strict inequalities, as long as -a_1(a_0+a_2)>0
%so is fine

\begin{theorem}[Lower Bound for Real Eigenvalues in case (i)]
\label{thm:main}
If $\mathbf{a}\in D_\mathbf{p}$ and  ${\mathbf{a}+k\mathbf{p}}\not{\in} \bar{D}_\mathbf{p}$ (equivalently $\widehat{\mathbf{a}+k\mathbf{p}}\not{\in} \bar{D}_\mathbf{p}$) for $k=1,2,..,n-1$ and
\begin{equation} 
\lambda^\dagger=\sqrt{-\rho_1(\rho_0+\rho_2)}
\label{eq:LowerBound}
\end{equation}
is real, there exist $\lambda_1,\lambda_2> \lambda^\dagger$ such that $\lambda_1$ is an eigenvalue of \eqref{eq:MainMatrixG} and $\lambda_2$ is an eigenvalue of \eqref{eq:MainMatrixZ}. 

Similarly if $\lambda^\dagger=\sqrt{-\rho_{n-1}(\rho_0+\rho_{n-2})}$ is real there exist $\lambda_1,\lambda_2> \lambda^\dagger$ such that $\lambda_1$ is an eigenvalue of \eqref{eq:MainMatrixG} and $\lambda_2$ is an eigenvalue of \eqref{eq:MainMatrixZ}. 
\end{theorem}
\begin{proof}
This follows from the previous three lemmas, letting $\rho_i=a_{i}$ and making note of \eqref{eq:insidenegative}. 
\end{proof}

Section \ref{sec:Classes} clarifies under what conditions $\lambda^\dagger$ is real, and Theorem \ref{thm:main} holds.
For example, when $\mathbf{p} = (1,1)^T$ then $\ab = (0,1)^T$ leads to a case (i) class, 
so Theorem~\ref{thm:realevs} holds, but the reality conditions in Theorem~\ref{thm:main} are 
not satisfied. It is also possible that $\lambda^\dagger=0$, so the real eigenvalue could be zero, 
so we also require that $\lambda^\dagger>0$. A sufficient condition on $\ab$ for this to hold is given in Lemma \ref{lem:RealitySatisfied}.

\subsection{Associated Eigenvector}
%NEW RE: 
We must also consider whether this eigenvalue is ``valid'' in the sense that the corresponding eigenfunction in the full system is in fact a function in $L^2(\mathbb{T}^2)$,
the square integrable function on the torus. For this we need to show that the Fourier coefficients are in $\ell^2$.
In this section we are going to work with the infinite dimensional system. The reason for this is that although the truncations are  useful in the calculation 
of eigenvalues, it is simpler to analyse the corresponding eigenvectors in the full system.
Thus our approach is to use the finite dimensional Zeitlin truncation to compute the eigenvalues, then take the limit as $N \to \infty$,
and then study the decay of the corresponding eigenvector.

Consider the linearised matrix for the full system
\begin{equation}
\label{eq:inflin}
M=
\begin{pmatrix}
\ddots & \vdots & \vdots & \vdots  &\vdots & \ddots \\
\cdots & 0 & \rho_{0} & 0 & 0 & \cdots \\
\cdots & -\rho_{-1} & 0 & \rho_1 & 0 & \cdots \\
\cdots & 0 & -\rho_0 & 0 & \rho_2 & \cdots \\
\cdots & 0 & 0 & -\rho_1 & 0 & \cdots \\
\ddots & \vdots & \vdots & \vdots  &\vdots & \ddots 
\end{pmatrix}.
\end{equation}
This is the limit in some sense of the matrices \eqref{eq:MainMatrixG} and \eqref{eq:MainMatrixZ} 
% or equivalently the limit of the Hamiltonian system
%  \eqref{eq:Antisymmetric}. 
% The corresponding Hamiltonian is  
which corresponds to the limit of the Hamiltonian systems $A=JS$ and $A'=J'S'$ with Hamiltonian 
$\mathcal{H}(\mathbf{\omega})=\sum_k \rho_k \omega_k^2$.

\begin{lemma}
\label{lem:eigenvector}
The infinite dimensional linearised system given by \eqref{eq:inflin} has a positive real eigenvalue $\lambda$ under the same conditions as Theorem \ref{thm:main}. 
% \footnote{HRD: is $\rho_k = 0$ allowed in the theorem 2?}
Furthermore, the associated eigenvector is in $\ell^2$. % eigenfunction of the full problem is square-integrable.
\end{lemma}

\begin{proof}

According to Theorem \ref{thm:main}, there exist positive real eigenvalues with some lower bound (independent of $N$) of \eqref{eq:MainMatrixZ} (and \eqref{eq:MainMatrixG}) for any $N$ and either choice of truncation. 
By taking $N\to\infty$ we can conclude there exists some positive real eigenvalue $\lambda$ of \eqref{eq:inflin}.  

Now consider an eigenvector associated with this eigenvalue, $\mathbf{v}=(...,v_{-1},v_0,v_1,v_2,...)^T$. For this to correspond to a real $L^2$ eigenfunction of the full problem (that is, for the Fourier series to converge), we need these Fourier coefficients to decay sufficiently fast: they need to be a sequence in $\ell^2$.

The entries of the (infinite dimensional) eigenvector $v_k$ of \eqref{eq:inflin} corresponding to eigenvalues $\lambda$ 
satisfy the recursion relation
\begin{equation}
	\lambda v_k= \rho_{k+1}v_{k+1}-\rho_{k-1}v_{k-1}.
\end{equation}

Since all $\rho_k \not = 0$ this can be rewritten as
\begin{equation}
	\label{eq:recurr}
	v_{k+1}=\frac{\lambda}{\rho_{k+1}}v_k+\frac{\rho_{k-1}}{\rho_{k+1}}v_{k-1}.
\end{equation}

Consider the limiting behaviour as $k\to\infty$. then $\rho_k\to\frac{1}{|\pp|^2}$. 
In this limit solutions to \eqref{eq:recurr} behave like solutions to
\begin{equation}
	\label{eq:recurr2}
	v_{k+1}={\lambda}|\pp|^2v_k+v_{k-1},
\end{equation}
see, e.g., \cite{Henry81}.

This linear recurrence has the general solution
\begin{equation}
	v_k = C_1\mu_1^k+C_2\mu_2^k
\end{equation}
where $C_1, C_2\in\mathbb{R}$ are constants and $\mu_1, \mu_2$ are solutions to $\mu^2-\lambda|\pp|^2\mu-1=0$. Thus $\mu_1\mu_2=-1$ and without loss of generality $|\mu_1|<1$, $|\mu_2|>1$ (note that we cannot have $|\mu_1|=|\mu_2|=1$ as $\lambda|\pp|^2\neq0$). 

Now as $\mathbf{v}$ is an eigenvector associated with a real eigenvalue, the span of the eigenvector is an invariant subspace of 
the Hamiltonian system with Hamiltonian $\mathcal{H}(\mathbf{\omega})=\sum_k \rho_k \omega_k^2$.
In fact, let $\mathbf{\omega}(0)=\mathbf{v}$, then $\mathbf{\omega}(t)=e^{\lambda t}\mathbf{v}$. As the Hamiltonian is an integral of the motion, $\mathcal{H}(\mathbf{v})=\mathcal{H}(e^{\lambda t}\mathbf{v})$. By taking the limit $t\to-\infty$, $\mathcal{H}(\mathbf{v})=\mathcal{H}(0)=0$. 
	
Therefore, 
\begin{equation}
	\mathcal{H}(\mathbf{v})=\sum_{k}\rho_k v_k^2=0;
\end{equation}
\begin{equation}
	\sum_{k\neq 0}\rho_k v_k^2=-\rho_0v_0^2.
\end{equation}

Now, if $C_2\neq 0$, 
\begin{equation}
	\sum_{k\neq 0}\rho_k v_k^2 \sim\sum_{k\neq 0}\rho_k(C_2\mu_2^{k})^2 \to \infty,
\end{equation}
recalling that $\rho_k\to\frac{1}{|\pp|^2}$ and $\rho_k>0$ for all $k\neq 0$. But $|\rho_0|<1$ and $v_k$ is finite, so there is a contradiction. Thus $C_2=0$ and $v_k=C_1\mu_1^k$ in the asymptotic limit, where $|\mu_1|<1$. This is exponential decay, which is sufficient for the Fourier series to converge.

Similarly for $k\to-\infty$, the limiting behaviour is governed by 
\begin{equation}
	\label{eq:recurr3}
	v_{k-1}=-{\lambda}|\pp|^2v_k+v_{k-1}.
\end{equation}

Again, this means $v_k$ is asymptotic to $C_1\mu_1^k+C_2\mu_2^k$ for $|\mu_1|<1$, $|\mu_2|>1$. By the same argument as above, we can conclude that $C_1=0$ and so $v_k=C_2\mu_2^k$ as $k\to\infty$.
Thus the Fourier coefficients decay exponentially on both sides with $|k|$, and hence $\mathbf{v}$ is in $\ell^2$.
% so the full eigenfunction associated with the eigenvalue in \ref{thm:main} for the full linearised system is $L^2$. 
\end{proof}

\section{Instability of Equilibria}
\label{sec:Classes}

\begin{figure}
\centering
\begin{tikzpicture}[scale=1.165]
\foreach \x in {0,...,4}
	\foreach \y in {0,...,3}
		\draw [fill] (\x,\y) circle [radius=.01];
		
\draw[<->] (0-0.3,0-0.5+0.33) to (3+0.3-2*3/5,5+0.5-2+0.33);
\draw[<->] (2-0.3,0-0.5) to (5+0.3-2*3/5,5+0.5-2);
\draw[<->] (4-0.3,0-0.5+0.33) to (5+0.3-1,5/3+0.5-1*5/3+0.33);

\draw  [dashed] (2+3/5,1) to (1,2);
\node at (2.0,1.7) {\tiny $\frac{2N+1}{|\mathbf{p}|}$};

\draw (2+3/5+3/40,1+5/40) to (2+3/5+3/40-5/40,1+5/40+3/40);
\draw (2+3/5-5/40,1+3/40) to (2+3/5+3/40-5/40,1+5/40+3/40);

\draw [fill, color=red] (1,2) circle [radius=.06];
\node [above left] at (1,2) {\tiny $\mathbf{a}$};
\draw [fill] (2,0) circle [radius=.06];
\draw (-0.5,-0.8) rectangle (4.5,4);

\end{tikzpicture}
\begin{tikzpicture}[scale=0.7]
\foreach \x in {-1,...,6}
	\foreach \y in {0,...,7}
		\draw [fill] (\x,\y) circle [radius=.01];

\draw [<->] (1-0.2*4/8,0-0.2) to (1+7*4/8+0.2*4/8,7+0.2);

\draw [fill, color=black] (1,0) circle [radius=.06] node [above left] {\tiny $\widehat{\mathbf{a}+k\mathbf{p}}$};
\draw [fill, color=red] (2,2) circle [radius=.06] node [above left] {\tiny $\mathbf{a}$};
\draw [fill, color=black] (3,4) circle [radius=.06] node [above left] {\tiny $\widehat{\mathbf{a}+(k+1)\mathbf{p}}$};
\draw [fill, color=red] (4,6) circle [radius=.06] node [above left] {\tiny $\mathbf{a}+\mathbf{p}$};

\draw (-1.5,-0.5) rectangle (6.5,7.5);

\end{tikzpicture}
\caption{For fixed  $\mathbf{p}$ and $\mathbf{a}$, there are some concerns with the wrapping of the Zeitlin truncation and the way this affects the intersection $\Sigma_\ab \cap D_\pp$. These situations are discussed in Lemma \ref{lem:NSequence}. \newline\hspace{\textwidth}
 \emph{Left:} the shortest distance between $\mathbf{a}$ and some non-consecutive $\widehat{\mathbf{a}+k\mathbf{p}}$ is at least $\frac{2N+1}{|\mathbf{p}|}$. As we are interested in the limit $N\to \infty$, this distance can be made arbitrarily large, so that  $\frac{2N+1}{|\mathbf{p}|}>2|\mathbf{p}|$. This ensures that $\mathbf{a}$ and $\widehat{\mathbf{a}+k\mathbf{p}}$ cannot both be in the disc $D_\mathbf{p}$.\newline\hspace{\textwidth}
 \emph{Right:} The situation where there exists $k\in\mathds{R}$ such that $\widehat{\mathbf{a}+k\mathbf{p}}$ lies on the line segment between $\mathbf{a}$ and $\mathbf{a}+\mathbf{p}$. This causes problems for our values of $\rho$. For $\widehat{\mathbf{a}+k\mathbf{p}}$ to lie at a lattice point on the line segment, $\text{gcd}(p_1,p_2)>1$. If $\text{gcd} (p_1,p_2)$ is odd, we can avoid this situation by choosing $N$ per Equation \eqref{eq:NChoice}; if $p_1$ and $p_2$ are both even this situation is unavoidable.} 
\label{fig:Problems}
\end{figure}
 
 To prove instability of an equilibrium for a given $\pp$ we need to find at least 
 one unstable class $\Sigma_{\ab}/\Sigma'_{\ab}$.
 A necessary condition for  a lattice point $\ab$ 
to lead to an unstable class is to be inside the unstable disc. 
More precisely we desire a lattice point $\ab$ that leads to a class 
of case (i) as this is the simplest situation for us to deal with.
Theorem \ref{thm:main} asserts that there is a real eigenvalue with an explicit lower bound under some certain conditions. The goal now is to determine for which $\pp$ there exists a lattice point $\ab$ such that the conditions of Theorem \ref{thm:main} are satisfied.

There are a number of other considerations when we take Zeitlin's truncation. Although this preserves the geometric structure and 
the Casimirs of the original problem, it introduces a problem that is not present in \cite{Li00}. 
This has already been mentioned in the classification of classes; it is the appearance of 
case (iii), in which a class intersects the unstable disc at non-consecutive points.
This may occur because now we have periodic boundary conditions not only in 
physical space, but also in Fourier space. There are two distinct problems caused by this, 
as illustrated in figure~\ref{fig:Problems}.

The lattice points of a class lie on parallel line segments with direction vector $\pp$ in the domain of Fourier modes. 
In the Zeitlin truncation there is more than one such line segment in the domain.
The first problem appears when the distance between these line segments 
is so small that more than one line segment intersects the unstable disc.
This can be fixed by making $N$ sufficiently large.
The second problem occurs when non-consecutive lattice points lie on the same 
line segment intersecting the unstable disc. If $\text{gcd}(p_1,p_2)$ is not even 
this can be fixed by choosing $N$ per \eqref{eq:NChoice}. 
Note that for our purposes, $\mathrm{gcd}(p_1,0)=p_1$ and $\mathrm{gcd}(p_1,p_2)=\mathrm{gcd}(|p_1|,|p_2|)$.

\begin{lemma}[Correct choices of $N$ for the Zeitlin truncation]%JW Version
 For all $\mathbf{p}= (p_1, p_2)^T$ such that $\kappa=\mathrm{gcd}(p_1,p_2)$ is not even, there exists a sequence of $N$ which increases without bound such that for all choices of $\mathbf{a}$ any two non-consecutive lattice points in $\Sigma'_\ab$ cannot both be in the unstable disc.
\label{lem:NSequence}
\end{lemma}
\begin{proof}
Let 
\begin{equation} 
  N=\frac{(2\tilde{N}+1)\kappa-1}{2}
  \label{eq:NChoice}
\end{equation}
for $A\in\mathds{N}$. Thus $2N+1=(2\tilde{N}+1)\kappa$. If $\kappa$ is not even, then such an $N$ is a positive integer and thus a valid grid size. Select as a lower bound $\tilde{N}>\frac{2|\mathbf{p}|^2-\kappa}{2\kappa}$ so that $N>\frac{2|\mathbf{p}|^2-1}{2}$. $\mathbf{p}$ is fixed and finite so this lower bound is always finite. We can thus find an infinite sequence of $N$ that increases without bound by letting $\tilde{N}$ increase without bound.

If $\mathbf{x}\in\Sigma_{\mathbf{a}}$ then $\mathbf{x}=\widehat{\mathbf{a}+k\mathbf{p}}$ for some $k\in\mathds{N}$. Thus $\mathbf{x}$ lies on the line parallel to the vector $\mathbf{p}$ that passes through the point $\ab+\mathbf{\Delta}_\mathbf{x} (2N+1)$ for some $\mathbf{\Delta}_\mathbf{x}\in\mathds{Z}^2$ (note that this point will be outside the domain $\mathcal{D}$). 

Similarly if $\mathbf{y}\in\Sigma_{\mathbf{a}}$ then it lies on the line parallel to $\pp$ that passes through $\ab+\mathbf{\Delta}_\mathbf{y} (2N+1)$ for some $\mathbf{\Delta}_\mathbf{y}\in\mathds{Z}^2$. Then the distance between these two lines is 
\begin{equation}
 d=\frac{|((\ab+\mathbf{\Delta}_\mathbf{x} (2N+1))-(\ab+\mathbf{\Delta}_\mathbf{y} (2N+1)))\times \mathbf{p}|}{|\mathbf{p}|}=\frac{(2N+1)|(\mathbf{\Delta}_\mathbf{x}-\mathbf{\Delta}_\mathbf{y})\times \mathbf{p}|}{|\mathbf{p}|}.
\end{equation}

$(\mathbf{\Delta}_\mathbf{x}-\mathbf{\Delta}_\mathbf{y})\times \mathbf{p}\in\mathds{Z}^2$, so $d=0$ or $d\geq\frac{2N+1}{|\pp|}$. If $d\geq\frac{2N+1}{|\pp|}$, this corresponds to $\mathbf{x}$ and $\mathbf{y}$ lying on different line segments, as in figure \ref{fig:Problems}. Thus the distance between two points on different line segments is at least $\frac{2N+1}{|\pp|}>2|\pp|$ for our choice of $N$.

If $d=0$, then $\mathbf{x}$ and $\mathbf{y}$ must lie on the same line segment. Thus $\mathbf{y}$ lies on the line parallel to $\pp$ passing through $\mathbf{x}$.

Write $\mathbf{p}=\kappa \mathbf{q}$ so $\mathbf{q}=(q_1,q_2)^T$ where $\text{gcd}(q_1,q_2)=1$. Then as $\mathbf{x},\mathbf{y}\in\mathds{Z}^2$,  $\mathbf{y}=\mathbf{x}+k\mathbf{q}$ for some $k\in\mathds{Z}$. But  $\mathbf{x},\mathbf{y}\in\Sigma_{\mathbf{a}}$ so $\mathbf{y}=\mathbf{x}+j\mathbf{p}+(2N+1)\mathbf{\Delta}$ for some $\mathbf{\Delta}\in\mathds{Z}^2$.

Thus $k\mathbf{q}=j\mathbf{p}+(2N+1)\mathbf{\Delta}$. So
\begin{equation}k\mathbf{q}=j\kappa\mathbf{q}+(2\tilde{N}+1)\kappa\mathbf{\Delta}=\kappa(j\mathbf{q}+(2\tilde{N}+1)\mathbf{\Delta}).\end{equation}
Thus $k\mathbf{q}$ is divisible by $\kappa$, but the elements of $\mathbf{q}$ both be divided by $\kappa$ by definition, so $\kappa | k$. Thus $\mathbf{y}=\mathbf{x}+\kappa\beta\mathbf{q}$ for some $\beta\in\mathds{Z}$, and so $\mathbf{y}=\mathbf{x}+\beta\pp$. Then if $|\mathbf{y}-\mathbf{x}|<2|\mathbf{p}|$ (the necessary condition for $\mathbf{x},\mathbf{y}$ both in $D_\pp$) this implies $\beta=0$ or $\beta=\pm 1$, and the result follows. \end{proof}

\begin{figure}
\centering
\begin{tikzpicture}

\path [pattern=north west lines] (6.95,7.05) to [out=125,in=-14] (4.84,8.51) to [out=52,in=-169] (6.43,9.57) to [out=-95,in=118] (6.95,7.05);

\path [pattern=north west lines] (7.05,6.95) to [out=125+180,in=-14+180] (9.16,5.49) to [out=52+180,in=-169+180] (7.57,4.43) to [out=-95+180,in=118+180] (7.05,6.95);

\draw  (7,7) circle [radius=sqrt(13)];

\draw [black,domain=-40:110] plot ({sqrt(13)*cos(\x)+4}, {sqrt(13)*sin(\x)+5});

\draw [black,domain=-40+180:110+180] plot ({sqrt(13)*cos(\x)+10}, {sqrt(13)*sin(\x)+9});

\draw [blue,dashed] (7,7) circle [radius=(sqrt(3)-1)*sqrt(13)];
\draw [->] (7,7) -- (10,9);

\draw [-] (7,7) -- (8.46410161514,9.19615242271) node [midway, above, sloped] (TextNode) {\tiny $(\sqrt{3}-1)|\mathbf{p}|$};

\draw [fill] (4,5) circle [radius=.04];
\draw [fill] (10,9) circle [radius=.04];

\draw [white,  fill] (5.85,8.73) circle [radius=0.56];
\draw [white, fill] (14-5.85,14-8.73) circle [radius=0.56];

\draw [red, very thick, pattern=north east lines, pattern color=red] (5.85,8.73) circle [radius=0.56];
\draw [red, very thick, pattern=north east lines, pattern color=red] (14-5.85,14-8.73) circle [radius=0.56];

\node [above right] at (10,9) {\tiny $\mathbf{p}$};

\node [right] at (7,7) {{\tiny $\;(0,0)$}};

\end{tikzpicture}

\caption{If $\mathbf{a}$ is inside the dashed
 blue circle centred at the origin, then $\rho_0+\rho_2<0$ and $\rho_0+\rho_{n-2}<0$. This circle has centre $(0,0)$ and radius $(\sqrt{3}-1)|\mathbf{p}|$. The shaded region here shows the overlap of this condition and the shaded region in figure \ref{fig:classes}. To show that there is at least one lattice point in the shaded regions we show the disc $D_c$ inscribed in this region (indicated by the small shaded red circles) has radius larger than $\sfrac{1}{\sqrt{2}}$.}

\label{fig:Conditions2}
\end{figure}
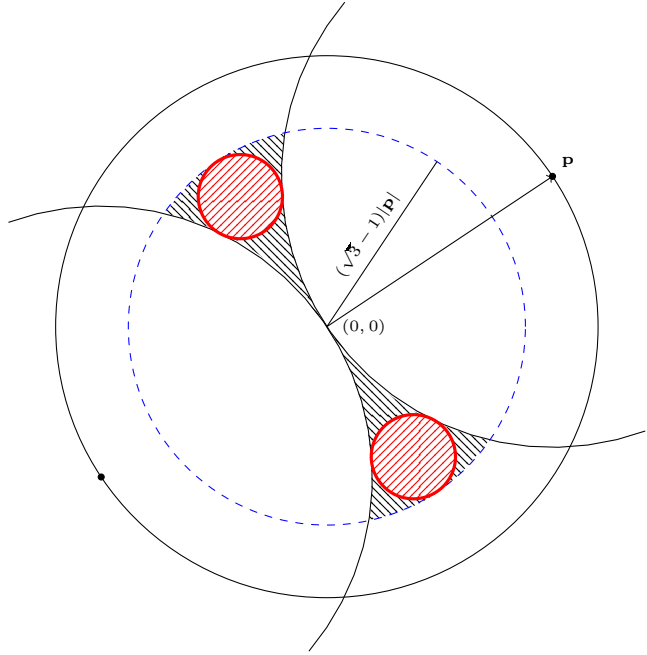

The outstanding issue with Zeitlin's truncation is that there is no appropriate choice of $N$ when $\gcd( p_1, p_2)$ is even. 
If $\kappa>1$ and $\gcd(\kappa,2N+1)=1$, then $\mathbf{p}$ will generate all multiples of $\mathbf{q}$ due to the wrapping operation. Thus classes that intersect the disc can return after leaving the disc and intersect the disc again, breaking the assumption of \ref{thm:LowerBoundZ} that $\rho_k<0$ for only one value of $k$. This behaviour continues for all values of $N$ with $\gcd(\kappa,2N+1)=1$. If $\kappa$ is even, this is true for any $N$; if $\kappa$ is odd, we select appropriate $N$ to avoid this.

It is important to note that this is not an error per se, but merely a failure of Lemma \ref{lem:NSequence} for our proof. The wrapping of the Zeitlin truncation associates modes in an artificial way, but still generates correct results. For instance if $\mathbf{p}=(6,2)^T=2(3,1)^T$, the class led by $\mathbf{a}=(1,6)^T$ intersects the unstable disc again at $(-2,5)^T=(1,6)^T-(3,1)^T$ for any finite truncation size. If we compare the non-imaginary eigenvalues of the class led by $\mathbf{a}=(1,6)^T$ with those of the Galerkin-truncated systems for $\mathbf{a}=(1,6)^T$ and $\mathbf{a}=(-2,5)^T$, the same eigenvalues are generated, with similar convergence to Figure \ref{fig:Convergence}.

Fortunately, this problem does not arise with the Galerkin truncation as the wrapping operation is omitted, and so the proof goes through without issue. For both truncations it is still necessary to establish whether for a 
given $\mathbf{p}$ there is an $\mathbf{a}$ such that the conditions of Theorem~\ref{thm:main} are met. We address that now.
\begin{lemma}
For all $\pp=(p_1,p_2)^T$ except $(1,0)^T, (1,1)^T, (1,2)^T$ (and permutations and sign changes thereof) there exists a choice of $\ab$ such that 
the reality conditions of Theorem \ref{thm:main} are satisfied for an appropriate choice of $N$ in the Galerkin truncation. Furthermore, if $\kappa =\mathrm{gcd}(p_1,p_2)$ is not even there is also a choice of $\ab$ such that the conditions of Theorem \ref{thm:main} are satisfied for the Zeitlin truncation.
\label{lem:RealitySatisfied}
\end{lemma}
\begin{proof}
For the bounds given in Theorems \ref{thm:LowerBoundG} and \ref{thm:LowerBoundZ}  to be real, positive, and hence a valid bound, we require $\rho_{0}<0$, $\rho_k\geq 0$ for all $k\neq 0$, and $\rho_0+\rho_2<0$ (or $\rho_0+\rho_{n-2}<0$). 

% For all $\mathbf{p}\neq (1,0)^t$, there exists at least one $\mathbf{a}\in D_{\pp}$ so that  $|\mathbf{a}|<|\mathbf{p}|$ and thus $\rho_0<0$ (for instance, $\mathbf{a}=(1,0)^t\in D_\pp$ for all $\mathbf{p}\neq (1,0)^t$).

% If $|\widehat{\mathbf{a}+k\pp}|<|\mathbf{p}|$ (and so $\rho_k<0$), then $|\mathbf{a}-\widehat{\mathbf{a}+k\pp}|<|\mathbf{a}|+|\widehat{\mathbf{a}+k\pp}|<2|\mathbf{p}|$.
% By Lemma \ref{lem:NSequence}, for an appropriate choice of $N$ this means that $\widehat{\mathbf{a}+k\pp}=\mathbf{a}\pm \mathbf{p}$. Thus $\rho_k \geq 0$ for all $k\neq 0,1,-1$. 

If $\rho_0<0,\;\rho_1,\rho_{-1}>0$, then $|\ab|<|\pp|$ and $|\ab\pm\pp|>|\pp|$. This is true if and only if $\ab$ is in the shaded region figure \ref{fig:classes}. 
In the Galerkin truncation, this is sufficient to show that $\rho_k>0$ for all $k\neq 0$.
By Lemma \ref{lem:NSequence}, in the Zeitlin truncation we can find an unbounded sequence of choices of $N$ such that this is sufficient to show $k\neq 0,\;\;\rho_k>0$. Thus for an appropriate choice of $N$ we only need to prove that there exists an $\ab$ such that $|\ab\pm\pp|>|\pp|$.

As $\lambda^\dagger=\sqrt{-\rho_1(\rho_0+\rho_2)}$ (or equivalently $\lambda^\dagger=\sqrt{-\rho_{-1}(\rho_0+\rho_{-2})}$) is required to be real and non-zero, and $\rho_{-1}>0$, then $\rho_0+\rho_{2}<0$ (equivalently $\rho_0+\rho_{-2}<0$).

If  $|\ab |<(\sqrt{3}-1)|\pp|$,  then $|\ab\pm 2\pp |\leq |\ab|+2|\pp| <(\sqrt{3}+1)|\pp| $. So 
%\begin{align}
%	\rho_0+\rho_{\pm 2}& = \frac{1}{|\pp|^2}-\frac{1}{|\ab|^2}+ \frac{1}{|\pp|^2}-\frac{1}{|\ab\pm 2\pp|^2}  \\
%			& < \frac{2}{|\pp|^2}-\frac{1}{|(\sqrt{3}-1)\pp|^2}-\frac{1}{|(\sqrt{3}+1)\pp|^2} \notag \\
%			&=0 \notag. \label{eq:Rho0Rho2}
%\end{align}
\begin{align}
	\rho_0+\rho_{\pm 2}& < \frac{2}{|\pp|^2}-\frac{1}{|(\sqrt{3}-1)\pp|^2}-\frac{1}{|(\sqrt{3}+1)\pp|^2} \\
			&=0 \notag. \label{eq:Rho0Rho2}
\end{align}

We thus need to show there exists some lattice point $\mathbf{a}$ such that $|\ab |<(\sqrt{3}-1)|\pp|$  and $|\ab\pm\pp|>|\pp|$. These two conditions are illustrated in figure \ref{fig:Conditions2} in the shaded region. Note that this is sufficient but not necessary for Theorem \ref{thm:main} to hold.

The idea now is to specify that $\ab$ is in the disc inscribed by the shaded region in figure \ref{fig:Conditions2}, which we call $D_c$. This disc is tangent to the circles with radii $|\pp|$ and centres $\pm \pp$ and the circle centred at the origin with radius $(\sqrt{3}-1)|\pp|$. It is a simple geometric exercise to show that such a circle has centre $\pm \frac{1}{\sqrt{3}}\begin{pmatrix} -p_2 \\ p_1 \end{pmatrix}$ and radius $\left ( \frac{2}{\sqrt{3}}-1 \right ) |\pp|$. 

If $\ab\in D_c$, it is outside the circle with centre $-\pp$ and radius $|\pp|$. Thus $\ab+\pp$ is outside $D_\pp$ and $\rho_1>0$. Similarly, $\rho_{-1}>0$. As $D_c$ is inside the disc with centre origin and radius $\left ( \frac{2}{\sqrt{3}}-1 \right ) |\pp|$ clearly $\rho_0<0$ and by the above $\rho_0+\rho_{\pm 2}<0$. If we take a Zeitlin truncation, choose appropriate $N$ such that  $\rho_k \geq 0$ for all $k\neq 0$ by Lemma 4. Then the conditions of Theorem \ref{thm:main} are satisfied and the resulting bound $\lambda^\dagger$ is real and positive.

All that remains is to show that there exists an integer lattice point $\ab\in D_c$. Any disc with a radius greater than $\frac{1}{\sqrt{2}}$ must contain some integer lattice point (as it wholly contains a square of side length $1$). Thus if
\begin{equation}
\label{eq:pcondition}
	|\mathbf{p}|>\frac{\sqrt{3}}{\sqrt{2}(2-\sqrt{3})}\approx 4.57.
\end{equation}
then the $D_c$ has radius greater than $\sqrt{2}$ and such a lattice point exists.
Note that this is a sufficient but not necessary condition on $\mathbf{p}$ . 

Checking the small number of $\mathbf{p}$ values with $|\mathbf{p}|<4.57$ and $\kappa $ odd there are appropriate lattice points $\ab$ for most such $\pp$. The following table shows an appropriate value for $\ab$ for most such $\pp$, and ``None'' where no such $\mathbf{a}$ exists.
\renewcommand\arraystretch{1.2}
\begin{equation} %very ugly formatting...
\begin{array}{|c | c | c | c | c | c | c | }
\hline 
\pp & (4,1)^T & (3,3)^T & (3,2)^T
	& (3,1)^T & (3,0)^T &
	(2,2)^T  \\
\hline 
\ab & (1,-2)^T & (1,-1)^T & (1,-2)^T
	& (1,-2)^T & (0,2)^T &
	(-1,2)^T   \\ 
\hline
\end{array}
\end{equation}

For reflections/rotations of these values of $\pp$ the corresponding reflection/rotation of $\ab$ is an appropriate choice.

Thus for all $\mathbf{p}=(p_1,p_2)^T$ such that $\kappa =\mathrm{gcd}(p_1,p_2)$ except
except $(1,0)^T, (1,1)^T, (1,2)^T$ and reflections and rotations of these, there is a choice of $\ab$ so that the conditions of Theorem \ref{thm:main} is satisfied when ${N=\frac{(2\tilde{N}+1)\kappa -1}{2}}$ for any $\tilde{N}>\frac{2|\mathbf{p}|^2-\kappa }{2\kappa }$.
\end{proof}

Now we can combine the results about eigenvalues and eigenvectors from the last section and the conditions 
on $\mathbf{p}$ when they are applicable in our main
\begin{theorem} \label{thm:fin}
The steady state $\Omega^*=\alpha \cos (\mathbf{p} \cdot  \mathbf{x} )+\beta \sin (\mathbf{p} \cdot  \mathbf{x} )$ is nonlinearly unstable for all $\mathbf{p}=(p_1,p_2)^T$ except $\mathbf{p}=\begin{pmatrix}\pm 1, 0\end{pmatrix}^T$, $\mathbf{p}=\begin{pmatrix} 0, \pm 1\end{pmatrix}^T$
and possibly $\mathbf{p}=\begin{pmatrix}\pm 1, \pm 1\end{pmatrix}^T$, $\mathbf{p}=\begin{pmatrix}\pm 2, \pm 1\end{pmatrix}^T$, $\mathbf{p}=\begin{pmatrix}\pm 1, \pm 2\end{pmatrix}^T$.
\end{theorem}
\begin{proof}
By Lemma \ref{lem:RealitySatisfied}, for all $\mathbf{p}$ except those listed above there exists some $\mathbf{a}$ such that $\rho_0<0$ and $\rho_0+\rho_2<0$ (or $\rho_0+\rho_{n-2}<0$) for an appropriate choice of $N$. Thus by Theorem \ref{thm:main} there exists a real positive eigenvalue $\lambda$. Moreover, the eigenvalue is greater than $\sqrt{-\rho_1(\rho_0+\rho_2)}$ (or $\sqrt{-\rho_{n-1}(\rho_0+\rho_{n-2})}$) which is both positive and independent of the choice of truncation size $N$. The truncation size $N$ can be increased without bound, by Lemma \ref{lem:NSequence}. Hence there is a hyperbolic eigenvalue in the limit $N\to\infty$ and the spectrum of the PDE is unstable. Now recall that any steady state $\Omega^*=\alpha \cos (\mathbf{p} \cdot  \mathbf{x} )+2\beta \sin (\mathbf{p} \cdot  \mathbf{x} )$ can be rewritten as $\Omega^*=2\Gamma\cos(\mathbf{p}\cdot\tilde{\mathbf{x}})$ and so the full result follows.

% Linear instability follows from this and the result in \cite{Li04}, which shows that the spectral mapping theorem holds. 

By Lemma \ref{lem:eigenvector}, the eigenvector associated with the eigenvalue $\lambda$ is in $\ell^2$.
The classes led by $\mathbf{a}$ and $-\mathbf{a}$ have the same eigenvalue, and the 
the corresponding eigenvectors can be combined to construct coefficients $\omega_{\mathbf{k}}$ 
of a real eigenfunction $\Omega_{\lambda}$ corresponding to $\lambda$. Since 
the eigenvectors $\mathbf{v}$ are in $\ell^2$ the periodic function $\Omega_\lambda$ is in $L^2$.
Together with the result in \cite{Li04}, which shows that the spectral mapping theorem holds, establishes linear instability.
To conclude nonlinear instability we refer to the work of \cite{Friedlander97} and \cite{shvidkoy03}. % , and Lin, Wang, and Zhang \cite{lin2014}. 
In \cite{Friedlander97} it was shown that sufficient conditions for nonlinear instability are linear instability together with a `spectral gap' condition. 
In \cite{shvidkoy03} it was shown that the essential spectrum of the linearised Euler operator in the cases we are considering is $i \mathds{R}$. 
Because of the presence of a point of discrete spectrum bounded away from the imaginary axis, we have a spectral gap, and hence nonlinear instability. 
%TODO perhaps these need to be reinstated??
%By nonlinear stability, we mean nonlinear stability in the $\mathcal^1(\mathds{T}^2}$ norm.  
%Noting that for the Euler equations linearised around these steady states, it was shown in [8] that the spectral mapping theorem holds, so we have that spectral stability implies linear stability in this case. Therefore we can conclude that these steady states are linearly unstable.
 
\end{proof}

Note that this does {\em not} preclude the possibility that the values of $\mathbf{p}$ listed as exceptions do not also lead to a linearly unstable steady state $\Omega^*$. 
%
%For instance, we have shown numerically that $\mathbf{p}=\begin{pmatrix}\pm 1, \pm 1\end{pmatrix}^T,\begin{pmatrix}\pm 2, \pm 1\end{pmatrix}^T,\begin{pmatrix}\pm 1, \pm 2\end{pmatrix}^T$ all correspond to unstable eigenvalues. 
In fact, for $\mathbf{p}=\begin{pmatrix}1, 1\end{pmatrix}^T,\begin{pmatrix}2, 1\end{pmatrix}^T$ and reflections/rotations thereof numerical results find non-zero non-imaginary eigenvalues. For $\mathbf{p}=\begin{pmatrix}1, 1\end{pmatrix}^T$ there is one complex quadruplet of eigenvalues, $\pm0.24822\pm 0.35172i$ to five decimal places, calculated with $N=1500$ and $\Gamma=1$.  For $\mathbf{p}=\begin{pmatrix}2, 1\end{pmatrix}^T$ there are two real pairs and two complex quadruplets of eigenvalues.
%Numerical results also show that for   $p_1\equiv p_2\equiv 0 \;(\text{mod }2)$, $\Omega$ is also unstable. 

Theorem~\ref{thm:fin} together with the numerical results mentioned and the spectral mapping theorem shown in [8] indicate that the only linearly stable 
equilibrium of type \eqref{eq:mainequilibria} is 
$\mathbf{p}=\begin{pmatrix}\pm 1, \pm 0\end{pmatrix}^T,\begin{pmatrix}\pm 0, \pm 1\end{pmatrix}^T$,
because only these have the single lattice point $\ab = (0,0)^T$ inside the unstable disc.
This exceptional case leads only to zero eigenvalues. 
All lattice points outside the unstable disc (including those on the boundary) do not contribute to instability
(see section 3.1) and hence this equilibrium is spectrally stable. By \cite{Kolev07} it is then linearly stable, and by \cite{Arnold98} and \cite{Li00} it is also Lyapunov stable.
%HRD: indicating that the Krein signatures are positive!

%We can compare this to the results described in \cite{Beck12}. For the Navier-Stokes equation with non-zero viscosity, they show that the solution
%\begin{equation}
%	\Omega(x,y,t)=e^{-\nu t}\cos (x)
%\end{equation}
%is `metastable', in the sense that it decays on a slow, viscous timescale. If we let $\nu=0$ (no viscosity), this is the same as our equilibrium with $\mathbf{p}=\begin{pmatrix}1,&0\end{pmatrix}^T$.

\section{Some Numerical Results}
\label{sec:numerical}

\begin{figure}
\centering
\includegraphics[width=0.7\textwidth]{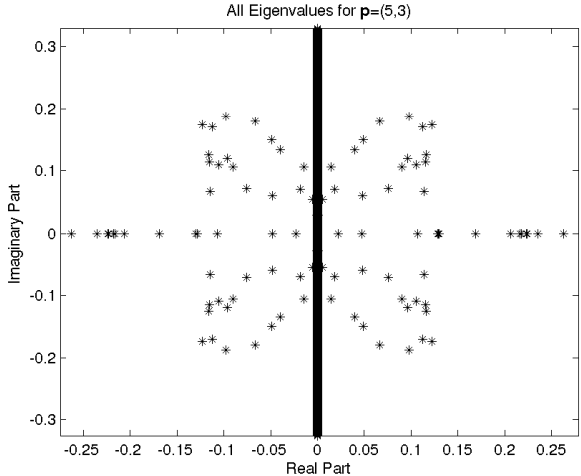}
\caption{All eigenvalues for the case $\mathbf{p}=(5,3)^T$, $\Gamma=\frac{1}{2}$. The Zeitlin truncation is used with  $N=200$. Note there are $200$ (or $100$ plus-minus pairs) eigenvalues with non-zero real part, and there are $100$ interior lattice points in $D_\mathbf{p}\setminus \{\mathbf{0}\}$. This confirms the result from \cite{Li04}. Of the non-imaginary eigenvalues, $56$ are real and $144$ are complex.  The number of interior points that satisfy $\rho_0<0$, $\rho_1,\rho_{-1}>0$ is $24$. All these points correspond to real pairs (two sets of $24=48$). The other real pairs come from interior points with $\rho_0<0$ and $\rho_1<0$ or $\rho_{-1}<0$. This usually creates a complex quadruplet but in a few cases corresponds to two real pairs instead. Increasing $N$ does not change the number of non-imaginary eigenvalues.} 
\label{fig:53Ensemble}
\end{figure}

\begin{figure}
\centering\includegraphics[width=0.49\textwidth]{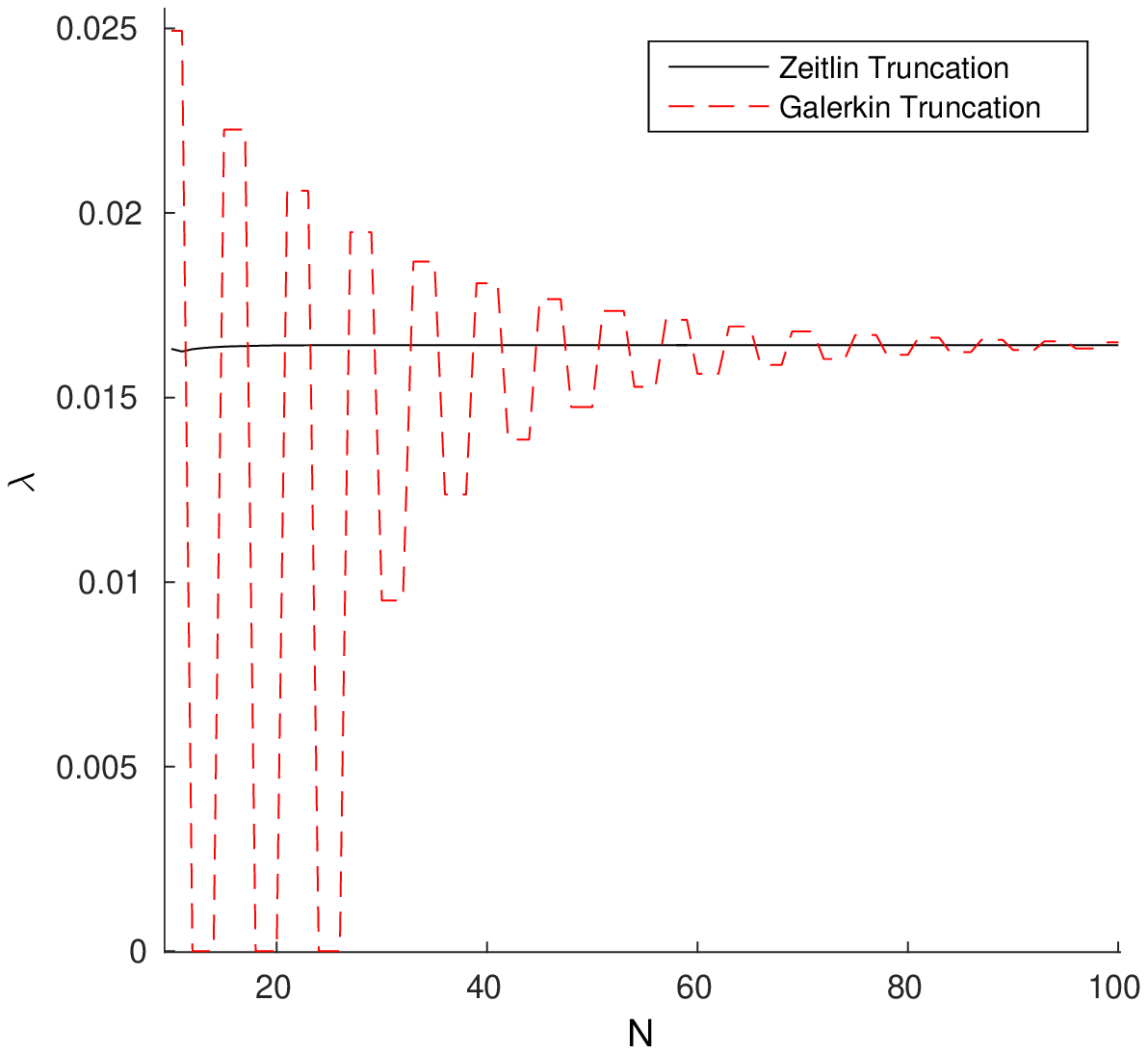}
\centering\includegraphics[width=0.49\textwidth]{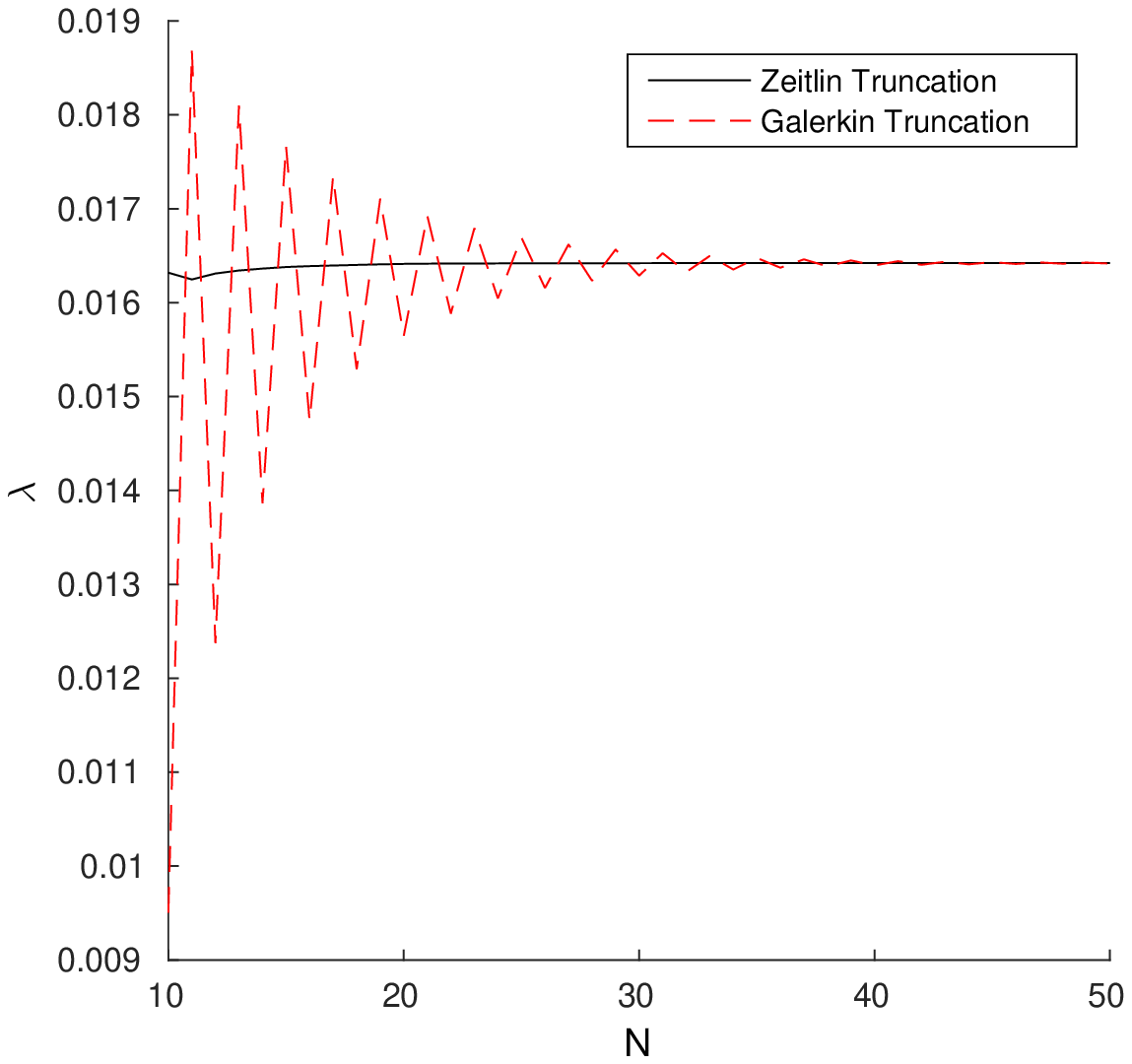}
\caption{Numerically computed real eigenvalues vs Fourier mode domain size $N$.For these figure $\mathbf{a}=(0, 3)^T$ and $\mathbf{p}=(3, 1)^T$. The red dashed lines shows the eigenvalues computed by the Galerkin truncation in equation \eqref{eq:MainMatrixG}, and the black solid lines show the eigenvalues computed by the Zeitlin truncation in equation \eqref{eq:MainMatrixZ}. \hl{For the left figure, the same truncation domain} \eqref{eq:trundomain} \hl{is used for the Zeitlin truncation and the Galerkin truncation, meaning a larger matrix is computer for the Zeitlin truncation. For the right figure, a Galerkin truncation with $2N+1$ modes was chosen so that the same number of Fourier modes are included in both calculations. The convergence of the eigenvalue as a function of $N$ computed with the Zeitlin truncation is significantly better in either case.
%Small values of $N$ are not shown, as $\mathcal{D}$ must be sufficiently large to avoid the concerns raised in Section~\ref{sec:Classes}. 
These plots} omit the factor of $\alpha/\alpha'$ for clarity.}
\label{fig:Convergence}
\end{figure}

\begin{figure}
\centering
\includegraphics[width=0.48\textwidth,height=6cm]{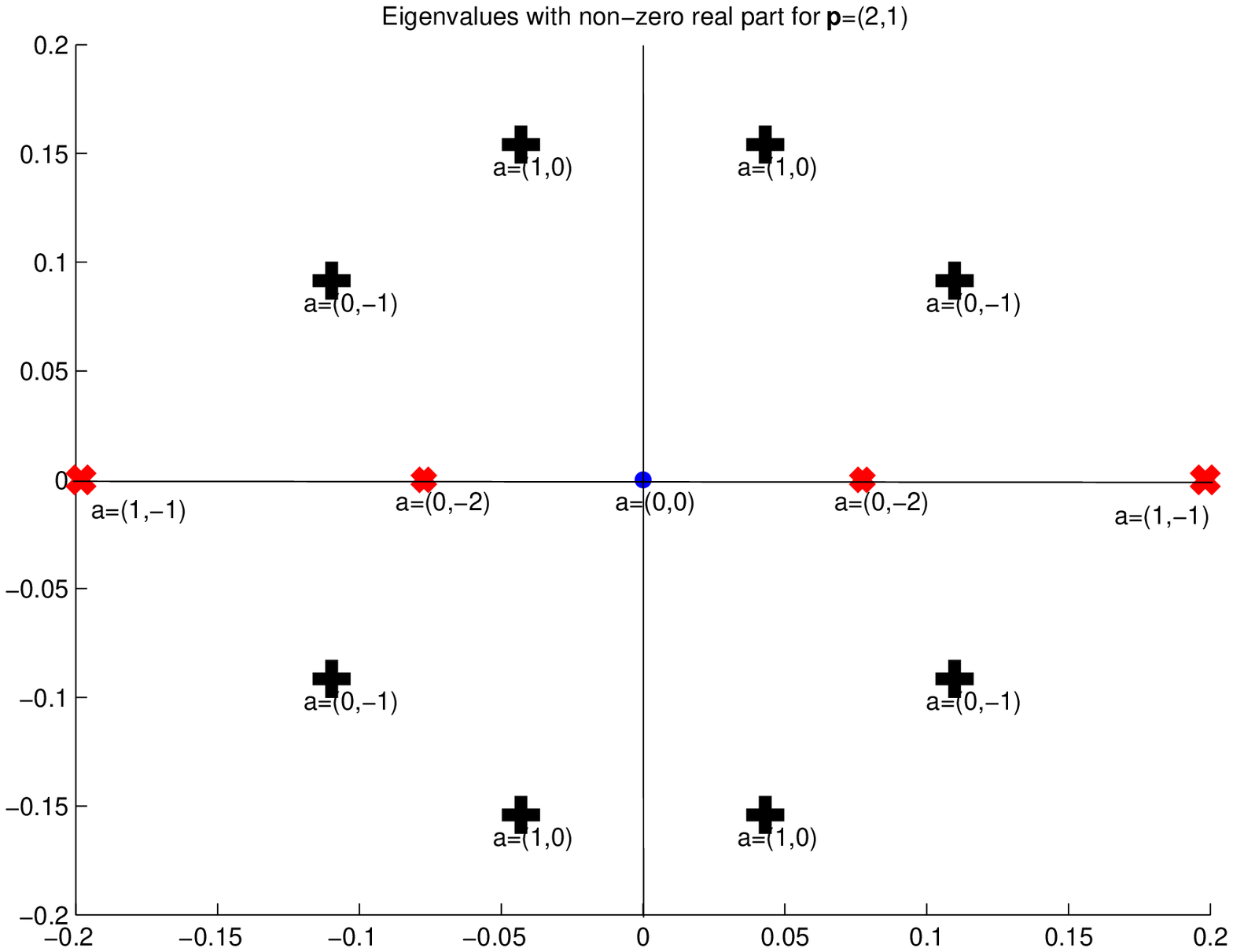}
\includegraphics[width=0.48\textwidth]{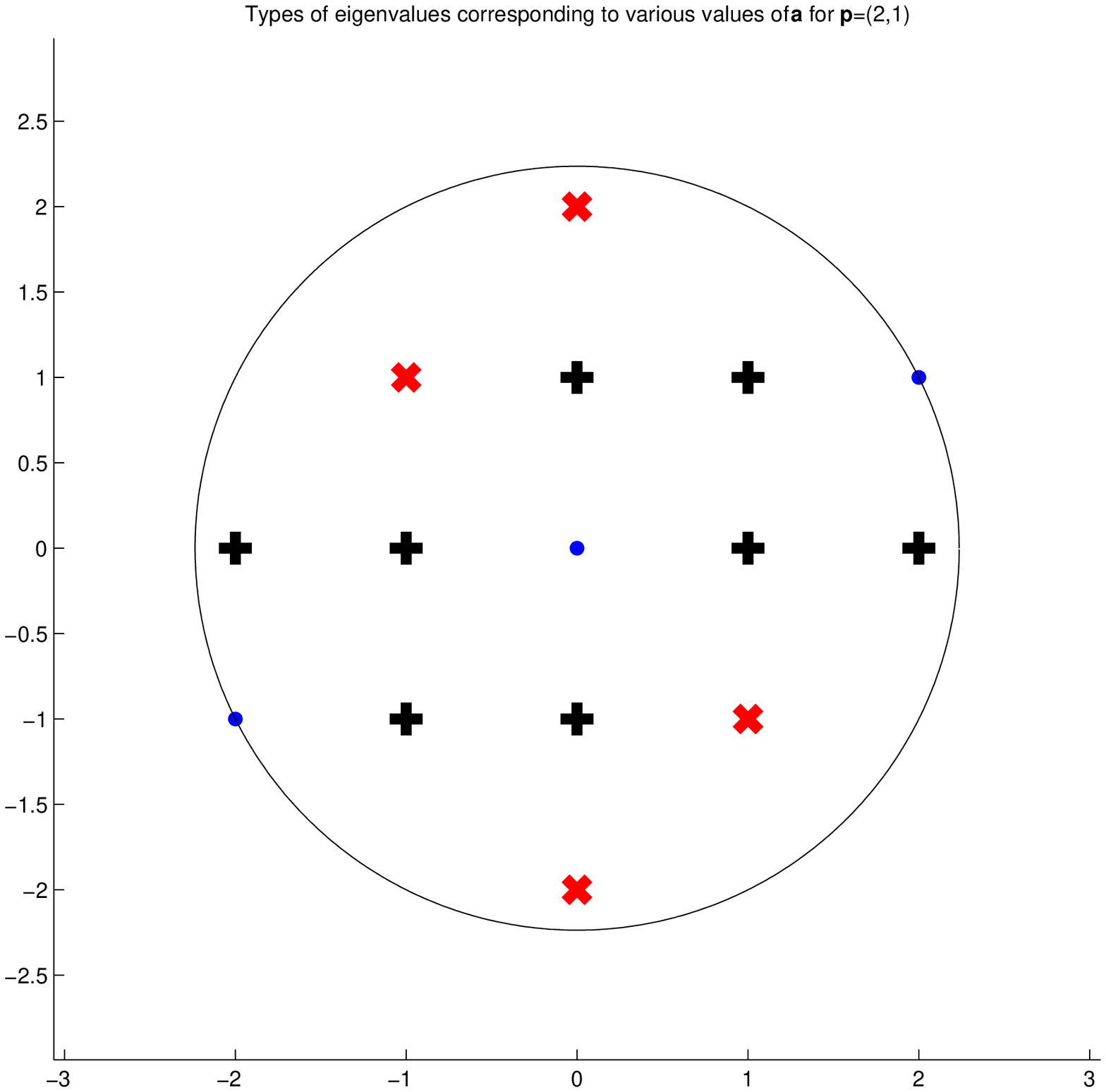}
\caption{Non-imaginary eigenvalues (left) and the corresponding lattice points $\mathbf{a}$ in the unstable disc (right) for the equilibrium with $\mathbf{p}=(2,1)^T$. At the top, the eigenvalues with non-zero real part for every unstable class $\Sigma_\mathbf{a}$ are shown. Eigenvalues with zero imaginary part are marked with a red $\times$, complex eigenvalues are marked with a black $+$, and the zero eigenvalues are marked with a blue dot. In the bottom figure we see the values of $\mathbf{a}$ that correspond to these classes. The zero class led by $\mathbf{a}=(0,0)^T$ gives only zero eigenvalues. Compare the locations of the classes corresponding to real eigenvalues to the shaded region in figure \ref{fig:classes}. For these figures, $\Gamma=\frac{1}{2}$, so $\Omega^*=\cos (\mathbf{x}\cdot\mathbf{p})$, and a Zeitlin truncation is used for the approximation.} 
\label{fig:EigTypes}
\end{figure}

%\begin{figure}
%\centering
%\includegraphics[trim=1cm 0cm 1cm 0cm, clip=true, width=0.7\textwidth]{Images/lbound_v_real_2.eps}
%\caption{The lower bound given by \eqref{eq:LowerBound} (black $*$) vs the maximum calculated eigenvalue (red $+$) for a variety of values of $\mathbf{p}$. We look at $\mathbf{p}=(p_1,1)^T$, $\mathbf{a}=(-1,\lfloor 0.5p_1 \rfloor)$ and vary $p_1$ along the $x$ axis. These values of $\mathbf{a}$ will be in the shaded region in figure \ref{fig:classes}, guaranteeing a real eigenvalue.  For the calculated eigenvalues a truncation size of $N=2000$ is used. For the lower bounds the choice of $N$ is irrelevant. The factor of  $\alpha$ is applied after taking the limit $\alpha \to \Gamma \mathbf{a}\times \mathbf{p}$. $\Gamma=1$ for this figure, and a Zeitlin truncation is used for the approximation.} 
%\centering
%\label{fig:LowerBoundVsCalculated}
%\end{figure}

%\begin{figure}
%\centering
%\includegraphics[width=0.8\textwidth]{Images/dvseigsrealp7_5N1000_with_im_2.eps}
%\caption{The real part of all eigenvalues with non-zero real part, plotted against the perpendicular distance of the class from $\Sigma_\mathbf{0}$. The eigenvalues with zero imaginary part are marked in blue, and the eigenvalues with non-zero imaginary part are marked in black. For this plot $\mathbf{p}=(7,5)^T$, $\Gamma=1$, $N=1000$, and a Zeitlin truncation is used for the approximation..}
%\label{fig:BoundDistance}
%\centering
%\end{figure}

\begin{figure}
\centering
\includegraphics[width=0.7\textwidth]{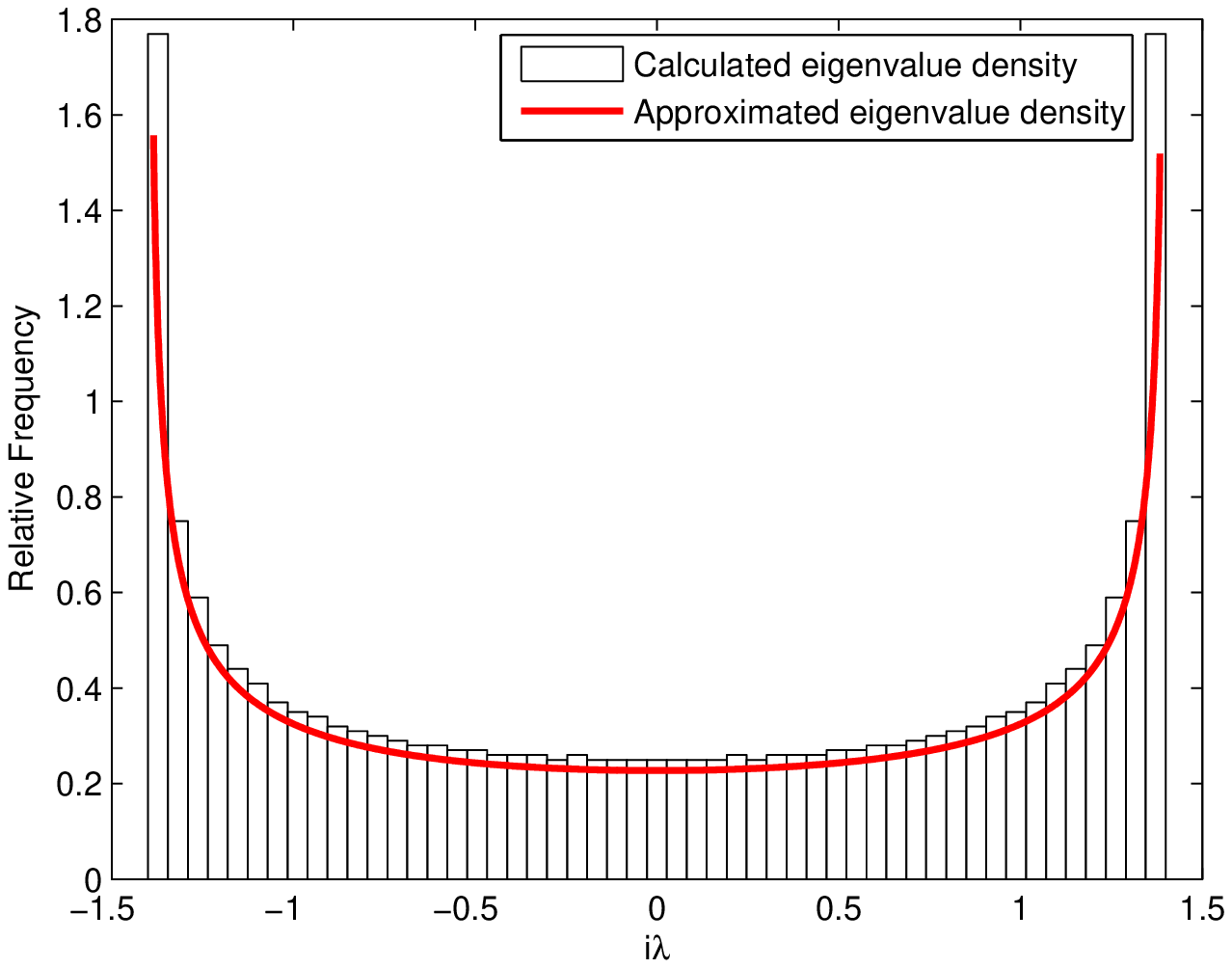}
\caption{The density \hl{of the} imaginary part of the spectrum, for the class $\mathbf{p}=(3,1)$, $\mathbf{a}=(1,-2)$, $N=1000$, and $\Gamma=0.5$. The bars show the normalized density of the imaginary eigenvalues computed for $N=1000$. The thick red line shows the approximate density computed by taking the approximation $\rho_k \to \frac{1}{|\mathbf{p}|^2}$. For this figure a Zeitlin truncation is used for the approximation.}
\label{fig:ImagDensity}
\centering
\end{figure}

\subsection{The Unstable Spectrum}
The Zeitlin class decomposition means that we now typically compute the eigenvalues of $(2N+1)$ matrices, each of size $(2N+1)\times (2N+1)$. Without the class decomposition, the eigenvalues of one $(2N+1)^2\times (2N+1)^2$ matrix need to be computed. So the class decomposition results in an extremely significant saving of computation time. For a Galerkin truncation, this computational saving is even more pronounced. However, this is at the expense of accuracy (see figure \ref{fig:Convergence}).

Figure \ref{fig:53Ensemble} shows all the eigenvalues associated with a fixed value of $\mathbf{p}$ and $N$. 
There are exactly twice the number of interior lattice points in $D_\mathbf{p} \setminus {\mathbf{0}}$. This agrees with the result in \cite{Li04} that the discrete spectrum of the corresponding operator has at most $2|D_\mathbf{p}|-2$ non-imaginary eigenvalues. Our numerical results indicate that this bound is likely to be sharp; for all choices of $\mathbf{p}$ tested there is equality.

Figure \ref{fig:Convergence} shows the values at which the calculated eigenvalues converge as a function of the size $N$ of our truncation domain $\mathcal{D}$. Compared to the Galerkin truncation the eigenvalue converges for much smaller values of $N$ when using Zeitlin's method.

Figure \ref{fig:EigTypes} shows the correspondence between the location of values of $\mathbf{a}$ and the types of eigenvalues of the class $\Sigma'_\mathbf{a}$. This corresponds to the results of Section \ref{sec:Instability}.  Compare the positioning of the Fourier modes with figure \ref{fig:classes}. 

%Figure \ref{fig:LowerBoundVsCalculated} shows the lower bound \eqref{eq:LowerBound} and the numerically calculated maximum real eigenvalue across a range of values of $\mathbf{p}$. For this figure we vary $p_1$ and let $\pp=(p_1,1)^T$ and $\ab=(-1,\lfloor \sfrac{p_1}{2} \rfloor )^T$. It appears that our lower bound is a reasonable approximation to the actual eigenvalue but not particularly close. This is fine, as the bound is only intended to be greater than zero. 
%
%Figure \ref{fig:BoundDistance} shows the real parts of the eigenvalues as a function of the distance of the class $\Sigma_{\mathbf{a}}$ from the class $\Sigma_{\mathbf{0}}$ for a fixed $\pp$ and $N$. Of particular note is that the real parts of the complex eigenvalues seem to have a relatively smooth correlation with this distance.
%The eigenvalues with largest real parts roughly correspond to lattice points $\ab$ 
%that are partway between $\Sigma_\mathbf{0}$ and the boundary of $D_\pp$.
%
%

\subsection{The Stable Spectrum}

Figure \ref{fig:ImagDensity} shows the density of the imaginary parts of the spectrum for $\mathbf{a}=(-4,7)^T$, $\mathbf{p}=(7,5)^T$. There are also non-imaginary eigenvalues but these are not shown on the figure. 

For any $\varepsilon>0$, we can choose sufficiently large $N$ so that there exists some $\ab$ such that $\frac{1}{|{\mathbf{b}}|^2}<\varepsilon$ for all $\mathbf{b}\in\Sigma'_\mathbf{a}$. So the imaginary spectrum of this class can be approximated by taking $\rho_k\approx \frac{1}{|\mathbf{p}|^2}$. The resulting matrix $A$ from \eqref{eq:MainMatrixZ} is now circulant. A circulant matrix is diagonalised by a discrete Fourier transform (see \cite{Karner02}). Thus the eigenvalues of $A$ are then found to be
\begin{equation}
	\lambda_j=\frac{2i}{|\mathbf{p}|^2} \sin \left (\frac{2\pi j}{n} \right ) \text{ for }j=0,...,n-1 \text{ where } j \text{ is the size of }A.
	\label{eq:circulantevs}
\end{equation}
See \cite{Karner02} for details of this calculation. 

Thus the approximate imaginary spectrum of $\Sigma_\mathbf{a}$ for sufficiently large $|\mathbf{a}|$ lies in the interval $ \frac{2i}{|\mathbf{p}|^2}[-|\alpha|,|\alpha|]$ on the imaginary axis. Taking the limit $N\to\infty$ (and so $n\to\infty$), for each $x\in[0,1]$ there is a correspondence with an eigenvalue $\lambda_x$ where $x=\frac{1}{2\pi}\sin^{-1}\left (\frac{\lambda_x|\mathbf{p}|^2}{2i}\right )$. Differentiating this gives the density function
\begin{equation}
	F(x)=\frac{|\mathbf{p}|^2}{\pi\sqrt{ 4\alpha ^2 - |\mathbf{p}|^4x^2 }}.
	\label{eq:appdensity}
\end{equation}

That is, the proportion of the eigenvalues lying between $c_1i$ and $c_2i$ on the imaginary axis is $\int_{c_1}^{c_2} F(x) \mathrm{d}x$ for $c_1,c_2\in\frac{2}{|\mathbf{p}|^2}[-|\alpha|,|\alpha|]$. This curve is also plotted in \ref{fig:ImagDensity}, and it agrees well with the numerically calculated eigenvalues. This is surprising as the value of $|\mathbf{a}|$ is not particularly high. We can conclude that equation \eqref{eq:appdensity} gives a reasonable approximation of the imaginary spectrum for many choices of $\mathbf{a}$.

 \cite{Li04,shvidkoy03} describe the essential spectrum of the linearised operator that coincides with our limit $N\to\infty$. The essential spectrum for the class led by $\ab$ is given in that paper as
\begin{equation}
\label{eq:essspec}
	\sigma_{\text{ess}}=i[-|\beta|,|\beta|],\text{ where }\beta=\frac{2}{|\mathbf{p}|^2}(\mathbf{a}\times\mathbf{p})\Gamma.
\end{equation}

In the limit $N\to \infty$, $\frac{\sin (\varepsilon\mathbf{a}\times \mathbf{p})}{\varepsilon}\to\mathbf{a}\times\mathbf{p}$, so $\frac{2}{|\mathbf{p}|^2}|\alpha|\to|\beta|$. Thus our approximation for large $N$ reproduces the essential spectrum of a single class calculated in \cite{Li04}. Note that this is the essential spectrum associated with a single subsystem of the linearised problem. The essential spectrum of the full system is the superposition of all these essential spectra. It was shown in \cite{shvidkoy03} that this is $i\mathds{R}$, which follows from considering the superposition of \eqref{eq:essspec} for all possible values of $\ab$.

\section{Conclusion}
\label{sec:Conclusion}

We have demonstrated the non-linear instability of the stationary solutions with vorticity $\Omega^*=2\Gamma\cos (\mathbf{p} \cdot  \mathbf{x} )$ in the Euler equations for almost 
all values of $\mathbf{p}$ such that $\mathbf{p}\neq (1,0)^T$ (or rotations and reflections thereof). 
We started out by using the Zeitlin truncation, and observe the numerical approximation of eigenvalues for finite $N$ obtained from Zeitlin's truncation 
converge much faster with $N$. However, the Zeitlin truncation does not behave well when $\text{gcd}(p_1,p_2)$ is even. Because of this, we developed much of our theory for both truncations.
%Moreover the behaviour of the eigenvectors in the Zeitlin truncation is more complicated than in the Galerkin truncation.  
%This is why in the end we used both truncations.

In addition, we have recreated and extended a number of results described by \cite{Li00}, \cite{Li04}, and \cite{shvidkoy03}. Specifically, we have shown that the ``unstable disc theorem'' presented in \cite{Li00} still holds true in the current context of finite dimensional approximation. 
Moreover, we have shown that for almost all $\mathbf{p}$ we can use the unstable disc to prove instability (as opposed to the existing stability results developed by Li and others).
We have also numerically verified the bound on the number of non-imaginary eigenvalues and the essential spectrum of an individual class in \cite{Li04} and \cite{shvidkoy03}. We used very different approaches and arguments to those papers.

There are obvious extensions of this work. The first is a complete description of the non-imaginary spectrum. This would require first showing that for two consecutive negative values of $\rho$, the corresponding subsystem has four non-imaginary eigenvalues, either two real pairs or a complex quadruplet. 
This would be a step towards proving that the bound from \cite{Li04} is sharp. Another extension would be to see if any of the methods used in this paper could be applied to more complex steady states, for instance $\Omega^*=\sin (p_1x)\sin(p_2y)$.

%It would also require to show that these and the eigenvalues described in this paper account for the entire non-imaginary spectrum. 
%This could be achieved by using the results of \cite{Li04}. % , or the Krein signature argument outlined in Section \ref{sec:Krein}.
 %{should I mention Krein stuff here??}

%Another extension would be a description, or an explanation, of the cases for  $\mathbf{p}$ that our method cannot analyse for the Zeitlin truncation. This could possibly arise from a different choice of algebra for the Poisson bracket or an alteration of Zeitlin's existing algebra. 
A similar analysis of the Euler fluid equations on a three dimensional torus would be interesting. \hl{A Zeitlin-style structure preserving truncation for the 3D case is not possible, as the Casimirs that make such a truncation useful are not present in the 3D problem}\cite{Zeitlin90}\hl{. Similar stability results may still be possible using a Galerkin style truncation instead.

There is also a discussion of a structure-preserving truncation that includes a viscosity term in }\cite{Zeitlin05}\hl{. A comparison of Zeitlin's truncation to standard truncations with a viscosity term included as in Figure }\ref{fig:Convergence} \hl{would be valuable.}

\hl{A significant extension of this material would be an application of the same methods to the Euler problem on a sphere. There similar structure preserving truncation also due to Zeitlin for the sphere }\cite{zeitlin2004self}\hl{ and so there is some hope of similar results in that setting. }

\section*{Acknowledgements} 
RM gratefully acknowledges M. Beck and Y. Latushkin for extremely helpful discussions on the known stability results on the 2D Euler equations.

%\bibliographystyle{jfm}
\bibliographystyle{plain}
\bibliography{draft_paper}   % name your BibTeX data base
%\bibliographystyle{plain}

%things for phd document
% The reduced matrix
%The Krein signature (which shows how the real eigenvalues become real at 0 only).
% More interpretation back to PDE
% A better justification for the Hamiltonian nature of the eigenvalues 
% Implications of writing the matrix as JS (DFT) - > Canonical Coordinates
% A discussion of rho critical
% The condition for \emph{any} non-imaginary eigenvalues
%Why does anyone worry about Gamma in mathds{C}???

\end{document}